\documentclass[reqno]{amsart}
\usepackage{amsmath,amssymb,amsthm}
\usepackage{mathrsfs}
\usepackage{verbatim}
\usepackage{mathtools}
\usepackage{color}
\usepackage[colorlinks=true, citecolor=blue, anchorcolor=red]{hyperref}
\usepackage{cancel}
\usepackage{enumerate}
\usepackage{tikz}

\numberwithin{equation}{section}

\newtheorem{teor}{Theorem}[section]
\newtheorem{lem}[teor]{Lemma}
\newtheorem{coro}[teor]{Corollary}
\newtheorem{propo}[teor]{Proposition}
\theoremstyle{definition}

\theoremstyle{remark}
\newtheorem{obser}[teor]{Remark}

\newcommand{\C}{{\mathbb C}}

\newcommand{\N}{{\mathbb N}}
\newcommand{\R}{{\mathbb R}}

\newcommand{\h}{\mathcal{H_{\gamma}}}
\newcommand{\eps}{\varepsilon}

\renewcommand{\epsilon}{\varepsilon}
\renewcommand{\Re}{\mathop{\text{\upshape{Re}}}}

\newcommand{\md}[1]{\color{black}#1\color{black}}
\newcommand{\mdb}[1]{\color{black}#1\color{black}}

\allowdisplaybreaks

\begin{document}

\title[Coupled thermoelastic plate-membrane system]{Long-time asymptotics for a coupled thermoelastic plate-membrane system}\thanks{Financial support through COLCIENCIAS via Project 121571250194.}

\author{Bienvenido Barraza Mart\'inez}
\address{B.\ Barraza Mart\'inez, Universidad del Norte, Departamento de Matem\'aticas y Estad\'istica, Barranquilla, Colombia}
\email{bbarraza@uninorte.edu.co}

\author{Robert Denk}
\address{R.\ Denk, Universit\"at Konstanz, Fachbereich f\"ur Mathematik und Statistik, Konstanz, Germany}
\email{robert.denk@uni-konstanz.de}

\author{Jonathan Gonz\'alez Ospino}
\address{J.\ Gonz\'alez Ospino, Universidad del Norte, Departamento de Matem\'aticas y Estad\'istica, Barranquilla, Colombia}
\email{gjonathan@uninorte.edu.co}

\author{Jairo Hern\'andez Monz\'on}
\address{J.\ Hern\'andez Monz\'on, Universidad del Norte, Departamento de Matem\'aticas y Estad\'istica, Barranquilla, Colombia}
\email{jahernan@uninorte.edu.co}

\author{Sophia Rau}
\address{S.\ Rau, Universit\"at Konstanz, Fachbereich f\"ur Mathematik und Statistik, Konstanz, Germany}
\email{Sophia.Rau@uni-konstanz.de}

\renewcommand{\shortauthors}{B. Barraza Mart\'inez et al.}

\begin{abstract}
In this paper we consider a transmission problem for a system of a
 thermoelastic plate with (or without) rotational inertia term coupled with a membrane with different variants of damping for the plate and/or the membrane.  We prove
 well-posedness of the problem and higher regularity of the solution and study the asymptotic behaviour of the solution, depending on the damping and on the presence of the rotational term.\\



\end{abstract}

\maketitle

\section{Introduction}
Transmission problems appear frequently in different fields of physics and engineering, e.g., in solid mechanics of  composed materials, in  processes of electromagnetism on ferromagnetic materials with dielectric constants, in the vibration of membranes, and  in coupled plates (see \cite{borsuk2010transmission}). Structures compound by a finite number of interconnected flexible elements, such as beams, plates, shells, membranes and combinations of them were studied in different settings in  \cite{ammari2010stabilization}, \cite {balmes2002tools}, \cite{denk2018exponential}, \cite{kleinman1988single}, \cite{lagnese1994modeling}, \cite{leissa1969vibration}, \cite{munoz2007asymptotic}, \cite{roy2006damping}.

The classical linear model of thermoelastic plates due to Kirchhoff  is given by
\begin{align*}
\rho_1 u_{tt}-\gamma\Delta u_{tt}+\beta_1\Delta^{2} u+\mu\Delta\theta & = 0
\quad\text{in} \quad\Omega_1\times(0, \infty),\\
\rho_0\theta_{t}-\beta_0\Delta\theta - \mu\Delta u_{t}  & =0 \quad\text{in}
\quad\Omega_1\times(0, \infty),
\end{align*}
where $\rho_0$, $\rho_1$, $\beta_0$, $\beta_1$ and $\mu$ are positive constants, and $\gamma\geq0$. Here, $\Omega_1\subset\R^2$ is the region occupied by the middle surface of the plate, $u$ represents the \md{transverse displacements of the points on the middle surface of } the plate, and $\theta$ is the difference of temperature on the plate with respect to a reference temperature. For the physical model and the deduction of the these equations, see  \cite{lagnese1988modelling}.

In the present paper, we study the situation where the thermoelastic plate in $\Omega_1$ is coupled with an elastic membrane in the region $\Omega_2$. To fix the geometric situation of the problem, we consider \md{two non-empty, open, connected and bounded subsets  $\Omega$ and $\Omega_{2}$ of $\mathbb{R}^{2}$,  with boundary of class $C^4$  such that $\overline{\Omega_{2}}\subset\Omega$}. We denote $\Omega_{1}\coloneqq \Omega\setminus\overline{\Omega_{2}}$, $\Gamma\coloneqq \partial\Omega$ and $I\coloneqq \partial\Omega_{2}$. Note that $\partial\Omega_{1}=\Gamma\cup I$. The plate-membrane system of interest is composed by a thermoelastic plate in $\Omega_1$, and a membrane, occupying in equilibrium  the region $\Omega_{2}$, as shown in Figure \ref{region2}.

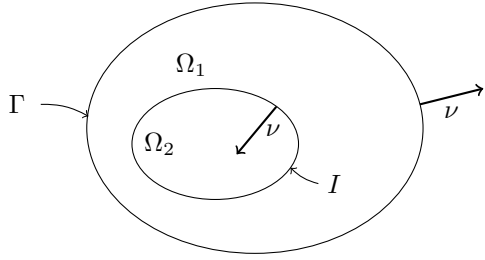
\begin{figure}[ht]
\begin{center}
\begin {tikzpicture}[x=3em,y=3em]
\draw (0,0) ellipse (1.05  and 0.7);
\draw (0.5, 0.2 ) ellipse (2.12 and 1.58);
\draw [->,bend left=15]  (-2.2, 0.5) to (-1.6,0.35);
\draw [->,bend left=15] (1.3, -0.5) to (0.95,-0.3);
\draw [->, thick] (0.77,0.47) to node{\quad\;$\nu$} +(-0.5,-0.6);
\draw [->, thick] (2.58,0.5) to node[below]{$\nu$} +(0.8,0.2);
\node at (-0.7,0) {$\Omega_2$};
\node at (-0.3, 1) {$\Omega_1$};
\node at (-2.5,0.5) {$\Gamma$};
\node at (1.5,-0.5) {$I$};
\end{tikzpicture}
\end{center}
\caption{Plate-Membrane system \label{region2}}
\end{figure}


Denoting by $u(x,t)$, $v(x,t)$ the vertical displacements of the points on the middle surface of the plate and on the membrane \md{with coordinates $x$ at time $t$}, respectively, and by $\theta(x,t)$ the temperature difference on the plate, the mathematical model for the structure is given by the equations
\begin{align}
\rho_{1} u_{tt}-\gamma\Delta u_{tt}+\beta_{1} \Delta^{2} u-\rho\Delta u_t+\mu\Delta\theta & =
0 \quad\text{in} \quad\Omega_{1} \times(0, \infty),\label{pl-1}\\
\rho_{0}\theta_{t}-\beta_{0}\Delta\theta-\mu\Delta u_{t}  & =0 \quad\text{in}
\quad\Omega_{1} \times(0, \infty),\label{pl-2}\\
\rho_{2} v_{tt} -\beta_{2} \Delta v  +mv_{t} & =0 \quad\text{in}
\quad\Omega_{2}\times(0, \infty),\label{m}
\end{align}
where $\rho_i$, $\beta_i$ ($i=0, 1, 2$) and $\mu$ are positive constants depending on the properties of the materials, and $\gamma$, $\rho$ and $m$ are non-negative constants. The coefficient  $\gamma>0$ represents the rotational inertia of the filaments of the plate and is proportional to the square of the plate thickness. Consequently, it is usual to consider this thickness very small (the case $\gamma=0$ corresponds to a thin plate). The coefficient $m\geq0$  describes the damping (or the absence of damping) for the wave equation \eqref{m}, whereas $\rho$ in \eqref{pl-1} describes a structural damping on the plate. We will also include the situation  when thermal effects for the plate are not taken into account by setting $\mu=0$ in \eqref{pl-1} and omitting \eqref{pl-2}.

We will assume that the plate is clamped at the exterior boundary $\Gamma$, namely,
\begin{equation}
\label{fron}u=\frac{\partial u}{\partial\nu}=0 \text{ on } \Gamma
\times(0, \infty),
\end{equation}
where $\nu$ represents the outward pointing unit normal vector to the boundary of $\Omega_{1}$ and consequently, $-\nu$  is the corresponding outward unit normal vector  to the boundary of $\Omega_{2}$.

Due to the lack of thermal effects in the membrane, we will assume that the difference of temperature in the interface is zero. We will also assume that the plate satisfies Newton's cooling law. This leads to the  boundary conditions
\begin{equation}
\label{temp}\theta=0 \text{ on } I \times(0, \infty) \quad\textup{ and }\quad \frac{\partial\theta}{\partial\nu}+ \kappa\theta=0 \mdb{\text{ on}}\  \Gamma
\times(0, \infty),
\end{equation}
for some constant $\kappa>0$.

In addition, we consider the following transmission conditions on the interface
\begin{align}
u=v,\quad\frac{\partial u}{\partial\nu}=0 &\quad\text{on}\ I \times
(0, \infty),\label{int-1}\\
\beta_{1} \frac{\partial\Delta u}{\partial\nu}+\beta_{2}\frac{\partial v}{\partial\nu}+\mu\frac{\partial\theta}{\partial\nu}=0 &\quad\text{on}\ I
\times(0, \infty),\label{int-2}
\end{align}
and the initial conditions
\begin{equation}\label{inc}
\begin{array}[c]{ccc}
u(\cdot, 0)=u_{0}, & v(\cdot, 0)=v_{0}, & \theta(\cdot, 0)=\theta_{0},\\
u_{t}(\cdot, 0)=u_{1}, & v_{t}(\cdot, 0)=v_{1}. &
\end{array}
\end{equation}
In \eqref{int-1}, the condition $u=v$ on $I$ is necessary for the continuity of the solution, and the condition $\frac{\partial u}{\partial\nu}=0$ on $I$ has the following interpretation: the transversal force caused by the tension and the one originated by the shear \md{stress } between the plate and the membrane cancel each other,  which forces   the horizontal displacements on the interface to be zero (compare with \cite{hassine2016asymptotic}).
The aim of the present paper is to study well-posedness, regularity,  and asymptotic behaviour   of the solution of \eqref{pl-1}-\eqref{inc}, in dependence on the parameters $m,\rho,\gamma$, and $\mu$.

For the case $\gamma=0$, we refer to \cite{ammari2010stabilization}, \cite{martinez2019regularity}, \cite{hassine2016asymptotic}, \cite{hernandez2005system} and \cite{liu2006existence}, where structures formed by a plate and a membrane were studied. In \cite{hernandez2005system}, the  author models a system composed by an elastic thin plate coupled with an elastic membrane and shows existence and uniqueness of weak solutions. In \cite{liu2006existence}, the authors study undamped plate-membrane systems, where the plate and the membrane are two layers occupying the same region in the plane. In \cite{ammari2010stabilization}, a coupled system of a wave equation and a plate equation with damping on the boundary without thermal effects is studied. In \cite{hassine2016asymptotic}, the \md{author } studies a transmission problem with the configuration presented in Figure \ref{region2}, but with the  plate being surrounded by the membrane. In \cite{martinez2019regularity} the plate is isothermal. Most of these references study some kind of stability for the solution but only a few of them deal with regularity.

Regarding the rotational inertia term ($\gamma>0$), we mention, e.g., \cite{avalos1998exponential}, \cite{chueshov2008attractors}, \cite{dell2013stability}, \mdb{\cite{FERNANDEZSARE20197085}}, \mdb{\cite{lasiecka2017global}}, \cite{lasiecka1993sharp}, \cite{rivera2004transmission} and \mdb{\cite{rivera2017large}}. From these, only in \cite{rivera2004transmission} a transmission problem is analyzed and it is of the thermoelastic plate-plate type. There seem to be few results for the structure \eqref{pl-1}-\eqref{inc}, even for the case $\gamma=0$.

The principal results of this work state the existence and uniqueness of the solution of the problem \eqref{pl-1}-\eqref{inc}, and its continuous dependence on the data (i.e., the well-posedness). It is also proved that the solution for the case $\gamma\geq0$ (i.e., with or without rotational inertia term) and $m\geq0$ (i.e., with or without damping over the membrane) has higher regularity. In particular, the boundary and transmission conditions
hold in the strong sense of  traces if the initial values are smooth enough.  Furthermore, we study the asymptotic behaviour of the solution in terms of the stability of the associated semigroup in different situations. For a damped membrane ($m>0$), we show that exponential stability holds if $\rho>0$ (Theorem~\ref{exp stab}) or if $\rho=\gamma=0$ (Theorem~\ref{Main_Th_section8}). For the undamped membrane ($m=0$), we have no exponential stability (Theorem~\ref{lack exp stab}), but polynomial stability (Theorem~\ref{teorcondition}) under some geometric condition.

The paper is organized as follows: In Section \ref{Well-posedness}, we define the basic spaces and operators and prove the generation of a $C_{0}$-semigroup of contractions, which implies the well-posedness of the problem \eqref{pl-1}-\eqref{inc}. In Section \ref{seccion_de_regularidad} we show  some spectral properties of the operator defined by the weak formulation of the transmission problem, and we also show  the regularity of the solution using the theory of parameter-elliptic boundary value problems and some ideas from \cite{martinez2019regularity}, where the authors study a plate-membrane transmission problem with $\gamma=0$ and without thermal effect over the plate.  In Section \ref{Exponential_stability} we prove exponential stability for the damped membrane, whereas in Section \ref{Lack_exponential_stability} we study the undamped membrane.

In the following, the letter $C$ stands for a generic constant which may vary in each time of appearance. We will also use the notation $\chi_A$  for the characteristic function of the set $A$, i.e. $\chi_A(x)=1$ for $x\in A$ and $\chi_A(x)=0$ else. If $X$ and $Y$ are Banach spaces, we write $X\subset Y$ for the continuous embedding, i.e. if $X$ is a subset of $Y$ and if $\operatorname{id}\colon X\to Y$ is continuous. The space of all bounded linear operators in $\mathcal H$ will be denoted by $\mathcal L(\mathcal H)$.


\section{Well-posedness}\label{Well-posedness}

Following the standard approach, we will formulate \eqref{pl-1}-\eqref{inc} as an abstract Cauchy problem and study the associated semigroup. We define $w = (w_j)_{j=1,\dots,5} \coloneqq (u,u_t,\theta,v,v_t)^\top$ and write \eqref{pl-1}--\eqref{m}, with the initial conditions \eqref{inc},  formally as a first order system
\begin{equation}
  \label{1-1}
   M(D) \partial_t w (t) - A(D) w(t) = 0\;(t>0),\quad w(0) = w_0
\end{equation}
with the diagonal matrix
\[ M(D) \coloneqq  \begin{pmatrix}
  1 & & & & \\ &\rho_1-\gamma\Delta & & & \\ & & \rho_0& & \\ & & &1 & \\ & & & &\rho_2
\end{pmatrix}\]
and, with
\[A(D) \coloneqq  \begin{pmatrix}
  0 & 1& 0&0 &0 \\ -\beta_1\Delta^2 &\rho \Delta &-\mu\Delta &0 & 0\\
  0&\mu\Delta & \beta_0\Delta &0 & 0\\0 &0 &0 &0 & 1\\ 0& 0& 0&\beta_2\Delta &-m
\end{pmatrix}.\]
and $w_0=(u_0,u_1,\theta_0,v_0,v_1)^\top$.\\
We start with the definition of the related operators in a weak Hilbert space setting. Here we have to distinguish the case $\gamma>0$ (presence of rotational inertia term) from the case $\gamma=0$. For $\gamma>0$, the Hilbert space $\mathcal H_\gamma$  is defined as the space of all complex-valued functions
\[ w = (w_1,\dots,w_5)^\top \in H^2(\Omega_1) \times H^1(\Omega_1) \times L^2(\Omega_1) \times H^1(\Omega_2)\times L^2(\Omega_2)\]
satisfying the boundary and transmission conditions
\begin{align*}
  w_1 = \partial_\nu w_1 = w_2 & = 0 \quad  \text{ on }\Gamma,\\
  w_1 = w_4,\, \partial_\nu w_1 & = 0 \quad \text{ on } I.
\end{align*}
For $\gamma=0$, we modify this definition by replacing the condition $w_2\in H^1(\Omega_1)$ by $w_2\in L^2(\Omega_1)$ and omitting the boundary condition $w_2=0$ on $\Gamma$. We endow the space $\mathcal H_\gamma$ for all $\gamma\ge 0$ with a scalar product which is adapted to the transmission problem. For $w,\phi\in \mathcal H_\gamma$, we set
\begin{align*}
  \langle w,\phi\rangle_{\mathcal H_\gamma} & \coloneqq  \beta_1 \langle\Delta w_1, \Delta\phi_1 \rangle_{L^2(\Omega_1)}+\rho_1\langle w_2, \phi_2\rangle_{L^2(\Omega_1)}+\gamma\langle\nabla w_2, \nabla\phi_2\rangle_{L^2(\Omega_1)}\\
  &   + \rho_0\langle w_3, \phi_3\rangle_{L^2(\Omega_1)}
  + \beta_2\langle\nabla w_4, \nabla\phi_4\rangle_{L^2(\Omega_2)} + \rho_2\left\langle w_5, \phi_5\right\rangle_{L^2(\Omega_2)}.
\end{align*}

\begin{lem}
  \label{2.1}
  For $\gamma>0$, the norm in $\mathcal H_\gamma$ is equivalent to the standard norm in $H^2(\Omega_1) \times H^1(\Omega_1) \times L^2(\Omega_1) \times H^1(\Omega_2)\times L^2(\Omega_2)$. For $\gamma=0$, the norm in $\mathcal H_0$ is equivalent to the standard norm in $H^2(\Omega_1) \times L^2(\Omega_1) \times L^2(\Omega_1) \times H^1(\Omega_2)\times L^2(\Omega_2) $.
\end{lem}

\begin{proof}
  Let $\gamma>0$. Obviously, all terms in the norm $\|\cdot\|_{\mathcal H_\gamma}$ can be estimated by the standard norm in $\mathcal H\coloneqq  H^2(\Omega_1) \times H^1(\Omega_1) \times L^2(\Omega_1) \times H^1(\Omega_2)\times L^2(\Omega_2)$, so we only have to show $\|\cdot\|_{\mathcal H_\gamma} \ge C \|\cdot\|_{\mathcal H}$.

  Let $w\in \mathcal H_ \gamma$. Then $w_1$ is a solution of the boundary value problem
  \begin{align*}
    \Delta u & = f \quad \text{ in } \Omega_1,\\
    \partial_\nu u & = 0 \quad \text{ on }\partial\Omega_1
  \end{align*}
  with $f\coloneqq \Delta u \in L^2(\Omega_1)$. By elliptic regularity (see \cite{Agranovich15}, Theorem~7.1.3), we get
  \begin{equation}
    \label{2-2}
    \| w_1\|_{H^2(\Omega_1)} \le C \big( \| \Delta w_1\|_{L^2(\Omega_1)} + \|w_1\|_{L^2(\Omega_1)}\big).
  \end{equation}
  Due to the boundary and transmission conditions $w_1=0$ on $\Gamma$ and $w_1=w_4$ on~$I$, the function $\chi_{\Omega_1}w_1 +  \chi_{\Omega_2} w_4$ belongs to $H^1_0(\Omega)$. With Poincar\'{e}'s inequality, we obtain
  \begin{equation}
    \label{2-1}
    \|w_1\|_{L^2(\Omega_1)}^2 + \|w_4\|_{L^2(\Omega_2)}^2 \le C \big( \|\nabla w_1\| _{L^2(\Omega_1)}^2 + \|\nabla w_4\|_{L^2(\Omega_2)}^2\big).
  \end{equation}
  As $w_1=0$ on $\Gamma$ and $\partial_\nu w_1=0$ on $\partial\Omega_1$, we can apply Poincar\'{e}'s inequality, integration by parts, and Young's inequality to see that for every $\epsilon>0$ there exists a $C_\epsilon >0$ such that
  \begin{align*}
    \|w_1\|_{L^2(\Omega_1)}^2 & \le C \|\nabla w_1\|_{L^2(\Omega_1)}^2 = -C \langle\Delta w_1, w_1\rangle_{L^2(\Omega_1)}\\
     & \le \epsilon C \|w_1\|_{L^2(\Omega_1)}^2 + C C_\epsilon \| \Delta w_1\|_{L^2(\Omega_1)}^2.
  \end{align*}
  Choosing $\epsilon$ with $\epsilon C \le \frac 12$, we can estimate $ \|w_1\|_{L^2(\Omega_1)} \le C\| \Delta w_1\|_{L^2(\Omega_1)}$. Combining this with \eqref{2-1} and \eqref{2-2}, we obtain
  \begin{equation*}
    \label{2-3}
    \| w_1\|_{H^2(\Omega_1)}^2 + \|w_4\|_{H^1(\Omega_2)}^2 \le C \big( \| \Delta w_1\|_{L^2(\Omega_1)}^2 + \|\nabla w_4\|_{L^2(\Omega_2)}^2 \big).
  \end{equation*}
  By definition of the norms in $\mathcal H_\gamma$ and the standard norm in $\mathcal H$, this yields $\|w\|_{\mathcal H}\le C \|w\|_{\mathcal H_\gamma}$.

  The same arguments show that also for $\gamma=0$ the norm in the space $\mathcal H_0 $ is equivalent to the standard norm in $ H^2(\Omega_1) \times L^2(\Omega_1) \times L^2(\Omega_1) \times H^1(\Omega_2)\times L^2(\Omega_2)$.
  \end{proof}

To formulate the transmission problem  \eqref{1-1} in a weak setting, we formally apply the operator $\beta_1\Delta^2$ to the first component and $-\beta_2\Delta$ to the fourth component. We obtain
\begin{equation*}
  \label{2-4}
   \widetilde M(D) \partial_t w (t) - \widetilde A(D) w(t) = 0\;(t>0),\quad w(0) = w_0
\end{equation*}
with
\[ \widetilde M(D) \coloneqq  \begin{pmatrix}
 \beta_1\Delta^2 & & & & \\ &\rho_1-\gamma\Delta & & & \\ & & \rho_0& & \\ & & &-\beta_2\Delta & \\ & & & &\rho_2
\end{pmatrix}\]
and
\[\widetilde A(D) \coloneqq  \begin{pmatrix}
  0 & \beta_1\Delta^2 & 0&0 &0 \\ -\beta_1\Delta^2 &\rho \Delta &-\mu\Delta &0 & 0\\
  0&\mu\Delta & \beta_0\Delta &0 & 0\\0 &0 &0 &0 & -\beta_2\Delta\\ 0& 0& 0&\beta_2\Delta &-m
\end{pmatrix}.\]
In this way, the weak formulation is adapted to the definition of the Hilbert space~$\mathcal H_\gamma$ and to the boundary and transmission conditions. Let $\mathcal H_\gamma'$ denote the antidual space of $\mathcal H_\gamma$, i.e. the space of all continuous conjugate linear functionals on $\mathcal H_\gamma$.  We define the operator $\mathbb M\colon \h \to \mathcal H_\gamma'$ by
\[ \langle \mathbb M w, \phi\rangle_{\mathcal H_\gamma'\times \h} \coloneqq  \langle w,\phi\rangle_{\h}\quad (w,\phi\in \h).\]
To define the operator related to $\widetilde A(D)$ in a weak setting, we introduce the space
\[ H^{2,1} \coloneqq  \{ (u,v)^\top \in H^2(\Omega_1)\times H^1(\Omega_2): u=\partial_\nu u =0 \text{ on }\Gamma, \, u=v,\, \partial_\nu u=0\text{ on } I\}\]
with inner product
\[ \langle (u,v), (u',v')\rangle_{H^{2,1}} \coloneqq  \beta_1\langle \Delta u,\Delta u'\rangle_{L^2(\Omega_1)} + \beta_2 \langle \nabla v,\nabla v'\rangle_{L^2(\Omega_2)}.\]
We have seen in the proof of Lemma~\ref{2.1} that this norm is equivalent to the standard norm in $H^2(\Omega_1)\times H^1(\Omega_2)$. Note that in the definition of $\h$, we have $(w_1,w_4)^\top\in H^{2,1}$. We define the subspace $\mathcal V\subset \h$ by
\[ \mathcal V \coloneqq  \{ w\in \mathcal H_\gamma: (w_2,w_5)^\top \in H^{2,1}, \, w_3\in H^1(\Omega_1),\, w_3=0\text{ on } I\}\]
with inner product
\begin{align*}
 \langle w,\phi\rangle_{\mathcal V} & \coloneqq  \big\langle (w_1,w_4), (\phi_1,\phi_4)\big\rangle_{H^{2,1}} +
\big\langle (w_2,w_5), (\phi_2,\phi_5)\big\rangle_{H^{2,1}} \\
& \quad + \beta_0\langle\nabla w_3,\nabla\phi_3\rangle_{L^2(\Omega_1)} + \beta_0 \kappa \langle w_3,\phi_3\rangle_{L^2(\Gamma)}.
\end{align*}
Now we can define $\mathbb A\colon \mathcal V\to \mathcal V'$ by
\begin{equation}\label{2-6}
\begin{aligned}
  \langle \mathbb Aw,\phi\rangle_{\mathcal V'\times \mathcal V} & \coloneqq  -\big\langle (w_1,w_4), (\phi_2,\phi_5)\big\rangle_{H^{2,1}} +
\big\langle (w_2,w_5), (\phi_1,\phi_4)\big\rangle_{H^{2,1}} \\
& - \mu \langle w_3,\Delta\phi_2\rangle_{L^2(\Omega_1)} + \mu \langle \Delta w_2, \phi_3\rangle_{L^2(\Omega_1)}  - \rho\langle\nabla w_2,\nabla\phi_2\rangle_{L^2(\Omega_1)} \\
& - \beta_0\langle\nabla w_3,\nabla\phi_3\rangle_{L^2(\Omega_1)} - \beta_0 \kappa \langle w_3,\phi_3\rangle_{L^2(\Gamma)} - m\langle w_5,\phi_5\rangle_{L^2(\Omega_2)}
\end{aligned}
\end{equation}
for $w,\phi\in\mathcal V$.

\begin{obser}
  \label{2.2}
  a) The norm in $\mathcal V$ is equivalent to the standard norm in $H^2(\Omega_1)\times H^2(\Omega_1)\times H^1(\Omega_1)\times H^1(\Omega_2)\times H^1(\Omega_2)$. In fact, we have already seen that the norm in $H^{2,1}$ is equivalent to the norm in $H^2(\Omega_1)\times H^1(\Omega_2)$, and for the component $w_3$ we have $\|w_3\|_{H^1(\Omega_1)}\le C\|\nabla w_3\|_{L^2(\Omega_1)}$ by Poincar\'{e}'s inequality and $\|w_3\|_{L^2(\Gamma)}\le \|w_3\|_{H^{1/2}(\Gamma)} \le C \|w_3\|_{H^1(\Omega_1)}$ by trace results.

  b) From the definition we immediately see that $\mathbb M \in L (\h,\mathcal H_\gamma')$ and $\mathbb A\in L(\mathcal V,\mathcal V')$. Moreover, $\mathbb M$ is defined as the scalar product in the Hilbert space $\h$ and therefore is an isometric isomorphism from $\h$ to $\mathcal H_\gamma'$.
\end{obser}

Based on Remark~\ref{2.2} b), we can define the $\h$-realization of the transmission problem as the operator $\mathcal A\colon \h \supset D(\mathcal A)\to \h$ by
\begin{equation*}\label{2-5}
D(\mathcal A) \coloneqq  \{ w\in\mathcal V: \mathbb A w \in\mathcal H_\gamma'\}, \quad \mathcal A w \coloneqq  \mathbb M^{-1}\mathbb A w.
\end{equation*}
 We consider the abstract Cauchy problem
\begin{equation}
   \label{pvi-f}
   \begin{aligned}
     \partial_t w(t) - \mathcal A w(t) & = 0 \quad (t>0),\\
     w(0) & = w_0
   \end{aligned}
\end{equation}
with $w_0 \coloneqq  (u_{0},u_{1},\theta_{0},v_{0},v_{1})^\top$. The following remark shows that this Cauchy problem is in fact the weak formulation of the transmission problem \eqref{pl-1}-\eqref{inc}.

\begin{obser}\label{2.3}
  a)  We have $\mathbb A w = \widetilde A(D) w $ for all  $w\in D(\mathcal A)$ and $\mathbb M   w = \widetilde M(D)  w$ for all $ w\in \h$  in the sense of distributions. This follows immediately from the definitions of the operators and integration by parts, when we choose $\phi\in \mathscr D(\Omega_1)^3 \times \mathscr D(\Omega_2) ^2$, where $\mathscr D(\Omega_1)$ stands for the infinitely smooth functions with compact support in~$\Omega_1$.
  Consequently, a function $w\in C^1([0,\infty); D(\mathcal A))$ is a classical solution of \eqref{pvi-f} if and only if $w$ satisfies \eqref{pl-1}--\eqref{m} in the distributional sense.

  b) Let $w\in D(\mathcal A)$ be of the higher regularity $w\in H^4(\Omega_1)\times H^2(\Omega_1)^2 \times H^2(\Omega_2)\times H^1(\Omega_2)$. Then $w$ satisfies the boundary and transmission conditions \eqref{fron}--\eqref{int-2} in the strong sense, i.e. as equality of the traces of the functions on $\Gamma$ and $I$, respectively.

  To see this, we only have to show that the second equality in \eqref{temp} and equality   \eqref{int-2} hold, as the other conditions are already included in the definition of $\mathcal V$.
    Setting $\phi=(0,0,\phi_3,0,0)^\top$, we obtain by a)
    \[ \langle \mathbb Aw,\phi\rangle_{\mathcal V'\times \mathcal V} = \langle \mu\Delta w_2 +\beta_0\Delta w_3, \phi_3\rangle_{L^2(\Omega_1)}.\]
  Comparing this with the definition of $\mathbb A$, we obtain, using integration by parts, that
    \[ \int_\Gamma (\kappa w_3 + \partial_\nu w_3) \overline{\phi_3} dS =0\]
   holds for all $\phi_3\in H^1(\Omega_1)$ with $\phi_3=0$ on $I$. Therefore, $\kappa w_3+\partial_\nu w_3=0$ holds on~$\Gamma$ in the strong sense, i.e. as equality in the trace space $H^{1/2}(\Gamma)$. In the same way, one can prove that \eqref{int-2} holds in the strong sense.
\end{obser}

To show well-posedness, we will also need the following result.

\begin{lem}\label{lema de densidad}
The space $\mathcal V$ is dense in $\h$, and therefore we have the dense embeddings
\[ \mathcal V\subset \h \subset (L^2(\Omega_1))^3\times (L^2(\Omega_2))^2 \subset\mathcal H_\gamma'\subset \mathcal V'.\]
\end{lem}

\begin{proof}
\textbf{(i)} In a first step, we show that
\[ V(\Omega_1)\coloneqq  \{ \phi\in H^{2}(\Omega_{1}): \phi=\partial_\nu \phi = 0 \textup{ on }\Gamma,\; \partial_{\nu}\phi=0 \textup{ on } I \} \]
is dense in $H^1_\Gamma(\Omega_1) \coloneqq  \{ u \in H^1(\Omega_1): u=0\text{ on } \Gamma\}$. For this, let
$ u\in H^{1}_{\Gamma}(\Omega_{1}) $. We choose a function $ \widetilde{\phi} \in C^{\infty}(\Omega_{1}) $ with $ \widetilde{\phi}=1 $ near $ I $,  $ \widetilde{\phi}=0 $ near $ \Gamma $, and $ 0\leq \widetilde{\phi} \leq 1 $ in~$ \Omega_{1} $. We set $ \phi\coloneqq \widetilde{\phi}^{2} $. Note that $ (1-\phi)u \in H^{1}_{0}(\Omega_{1})$ and $ \widetilde{\phi}u\in H^{1}(\Omega_{1}) $. As the test functions are dense in $H^1_0(\Omega_1)$, there exists a sequence $ \big(\phi_{n}^{(1)}\big)_{n\in\N}\subset \mathscr D(\Omega_{1}) $ such that $ \phi_{n}^{(1)}\to (1-\phi)u $ in $H^1(\Omega_1)$ for $n\to\infty$. Moreover, as the domain of the Neumann Laplacian
\[  D(\Delta_{N})\coloneqq  \{ u\in H^{2}(\Omega_{1}):\partial_{\nu}u=0   \textup{ on }   \partial\Omega_1 \} \]
is dense in $H^{1}(\Omega_{1})$ (see \cite{Ouhabaz05}, Lemma~1.25), there exists a sequence  $ \big(\widetilde{\phi}_{n}^{(2)}\big)_{n\in\N}\subset D(\Delta_{N}) $ with $ \widetilde{\phi}^{(2)}_{n}\to \widetilde{\phi}u $ in $ H^{1}(\Omega_{1}) $. Now, setting $ \phi_{n}^{(2)}\coloneqq \widetilde{\phi} \widetilde{\phi}_{n}^{(2)}\in V(\Omega_1)$ for $n\in\N$, we get  $\phi_{n}\coloneqq \phi_{n}^{(1)}+ \phi_{n}^{(2)} \to  (1-\phi)u + {\widetilde{\phi}}^2 u=u$ in $H^{1}(\Omega_{1})$.

\textbf{(ii)} Now we show that  $\mathcal V$ is dense in $ \h$. Comparing the definitions of $\mathcal V$ and $\h$ and noting that test functions are dense in $L^2$ spaces, we only have to consider the case $\gamma>0$ and to show that the embedding
\[ H^{2,1} \subset H^1_\Gamma(\Omega_1)\times L^2(\Omega_2)\]
is dense. Therefore, we fix $ u\in H^{1}_{\Gamma}(\Omega_{1}) $, $ v\in L^{2}(\Omega_{2}) $ and $ \varepsilon>0 $.
Using step (i), we find a function $ \phi_1\in V(\Omega_1) $ with $ \left\|u-\phi_1\right\|_{H^{1}(\Omega_{1})}<\frac{\eps}{2} $. Now, let $ \widetilde{\phi}_{1} \in H^{1}(\Omega_{2})$ such that $ \phi_1=\widetilde{\phi}_{1} $ on $ I $, and choose  $ \phi_2,\phi_3\in \mathscr D(\Omega_{2}) $ with $ \|\widetilde{\phi}_{1}-\phi_2\|_{L^{2}(\Omega_{2})}<\frac{\varepsilon}{4} $ and $ \left\|v-\phi_3\right\|_{L^{2}(\Omega_{2})}<\frac{\eps}{4} $. Then we obtain $ (\phi_1, \widetilde{\phi}_{1}-\phi_2+\phi_3)\in H^{2,1} $ and \[ \|(u, v)-(\phi_1, \widetilde{\phi}_{1}-\phi_2+\phi_3)\|_{H^1(\Omega_1)\times L^{2}(\Omega_{2})}<\eps .\]
Note that the embedding $\mathcal V\subset \h$ is dense and injective, and the same holds for the embedding of $\h$ into $(L^2(\Omega_1))^3\times (L^2(\Omega_2))^2$. Therefore, all embeddings stated in the lemma are dense.
\end{proof}

\begin{teor}\label{!}
For all $\gamma, \rho, m\geq0$, the operator $\mathcal A$ generates a $C_{0}$-semigroup $(\mathcal{S}(t))_{t\geq0}$ of contractions on $\h$. Therefore, for any $w^{(0)}\in D(\mathcal A)$ there exists a unique classical  solution $w\in C^1([0, \infty), \h)\cap C([0, \infty), D(\mathcal A))$ of \eqref{pvi-f}.
\end{teor}

\begin{proof}
Following a standard approach, we show that $\mathcal A$ is dissipative and $1-\mathcal A$ is surjective and apply the theorem of Lumer--Phillips.

Let $w\in D(\mathcal A)$. As $\mathbb M\colon \h\to \mathcal H_\gamma'$ is defined by $\langle \mathbb M w,\phi\rangle_{\mathcal H_\gamma'\times \h} = \langle w,\phi\rangle_{\h}$, we get
\begin{equation}\label{2-11a}
\langle \mathcal A w,w\rangle_{\h} = \langle \mathbb M^{-1}\mathbb Aw,w\rangle_{\h} = \langle \mathbb Aw,w\rangle_{\mathcal H_\gamma'\times \h}.
\end{equation}
By the definition of $\mathbb A$ in \eqref{2-6}, we immediately obtain
\begin{equation}\label{2-11}
\begin{aligned}
  \Re \langle \mathbb Aw,w\rangle_{\mathcal H_\gamma'\times \h} & = -\rho\|\nabla w_2\|^2_{L^2(\Omega_1) }
  - \beta_0\|\nabla w_3\|_{L^2(\Omega_1)}^2 - \beta_0\kappa \|w_3\|_{L^2(\Gamma)}^2 \\
  & -m\| w_5\|^2_{L^2(\Omega_2)} \le 0,
\end{aligned}
\end{equation}
 which shows that $\mathcal A$ is dissipative.

To show that $1-\mathcal A$ is surjective, it suffices by Remark~\ref{2.2} b) to show that $\mathbb M -\mathbb A\colon D(\mathcal A)\to \mathcal H_\gamma'$ is surjective. Let $f\in \mathcal H_\gamma'$. We have to find $w\in D(\mathcal A)$ such that
\begin{equation}
  \label{2-7}
  (\mathbb M -\mathbb A) w = (\widetilde M(D)-\widetilde A(D))w = f
\end{equation}
holds in $\mathcal H_\gamma'$ (cf. Remark~\ref{2.3}). From \eqref{2-7} we obtain
\begin{equation}
  \label{2-8}
  \begin{aligned}
    \beta_1 \Delta^2 w_1 & = \beta_1\Delta^2 w_2 +f_1,\\
    \beta_2 \Delta w_4 & = \beta_2 \Delta w_5 - f_4
  \end{aligned}
\end{equation}
as equality in $(H^{2,1})'$. Replacing this into \eqref{2-7}, we get
\begin{equation}\label{2-9}
 \begin{pmatrix}
  \rho_1 + \beta_1\Delta^2 - (\gamma+\rho)\Delta & \mu \Delta & 0 \\
  -\mu\Delta & \rho_0-\beta_0\Delta & 0 \\
  0 & 0 & \rho_2 + m-\beta_2\Delta
\end{pmatrix}\begin{pmatrix}
  w_2 \\ w_3 \\ w_5
\end{pmatrix}
= \begin{pmatrix}
  f_2-f_1 \\ f_3\\ f_5-f_4
\end{pmatrix} \eqqcolon \tilde f.
\end{equation}
We will solve this weakly with respect to the dual pairing $\mathcal V_0'\times \mathcal V_0$, where $\mathcal V_0$ is the projection of $\mathcal V$ to the components $(w_2,w_3,w_5)$, i.e.
\[\mathcal V_0 \coloneqq  \{(w_2,w_3,w_5)^\top: (0,w_2,w_3,0,w_5)^\top\in\mathcal V\}.\]
So we define the sesquilinear form $b\colon \mathcal V_0\times \mathcal V_0\to \C$ by
\begin{align*}
  b((w_2,&w_3,w_5),(\phi_2,\phi_3,\phi_5)) \coloneqq    \rho_1\langle w_2,\phi_2\rangle_{L^2(\Omega_1)} + \beta_1 \langle\Delta w_2,\Delta \phi_2\rangle_{L^2(\Omega_1)}\\
  &+ (\gamma+\rho) \langle\nabla w_2,\nabla \phi_2\rangle_{L^2(\Omega_1)}  - \mu\langle \nabla w_2,\nabla \phi_3\rangle_{L^2(\Omega_1)} \\
  &  + \mu\langle \nabla w_3,\nabla \phi_2\rangle_{L^2(\Omega_1)} +\rho_0\langle w_3,  \phi_3\rangle_{L^2(\Omega_1)} +\beta_0\langle\nabla w_3,\nabla\phi_3\rangle_{L^2(\Omega_1)}\\
  & + (\rho_2+m)\langle w_5,\phi_5\rangle_{L^2(\Omega_2)} +\beta_2\langle\nabla w_5,\nabla\phi_5\rangle_{L^2(\Omega_2)}.
\end{align*}
Obviously, $b$ is continuous, and a computation of $b((w_2, w_3,w_5),(w_2, w_3,w_5))$ shows that
\[ \Re b((w_2, w_3,w_5),(w_2, w_3,w_5)) \ge C \|(w_2, w_3,w_5)\|_{\mathcal V_0}^2. \]
As the right-hand side of \eqref{2-9} belongs to $\mathcal V_0'$, we may apply the theorem of Lax--Milgram to obtain a unique solution $(w_2,w_3,w_5)\in\mathcal V_0$ of
\[ b((w_2, w_3,w_5),(\phi_2,\phi_3,\phi_5)) = \tilde f ((\phi_2,\phi_3,\phi_5)).\]
By definition of $\mathcal V_0$ we have  $(w_2,w_5)\in H^{2,1}$. Because
\begin{equation}
\label{2-10} \begin{pmatrix}
  \beta_1 \Delta^2 & 0 \\ 0 & -\beta_2 \Delta
\end{pmatrix}\colon H^{2,1} \to (H^{2,1})'
\end{equation}
is an isomorphism due to Remark~\ref{2.2} b),   the right-hand side of \eqref{2-8} belongs to $(H^{2,1})'$. By the same reason, there exists a unique $(w_1,w_4)\in H^{2,1}$ such that \eqref{2-8} holds in $(H^{2,1})'$.

Altogether, we have found $w\in \mathcal V$ such that \eqref{2-7} holds in $\mathcal V'$, i.e.
\[ \big(   (\mathbb M-\mathbb A) w\big) (\phi) =  f (\phi) \quad (\phi\in \mathcal V).\]
As the right-hand side belongs to $\mathcal H_\gamma'$ and $\mathcal V$ is dense in $\h$ by
Lemma~\ref{lema de densidad}, also the left-hand side belongs to $\mathcal H_\gamma'$, and \eqref{2-7} holds in $\mathcal H_\gamma'$. In particular, $\mathbb A w = \mathbb M w - f\in \mathcal H_\gamma'$, which shows that $w\in D(\mathcal A)$. Therefore, $1-\mathcal A$ is surjective, and an application of the theorem of Lumer--Phillips finishes the proof.
\end{proof}

\section{Spectral properties and regularity of the solution}
\label{seccion_de_regularidad}

In this section, we study properties of the spectrum of the operator $\mathcal A$ defined above and show that functions in   its domain have higher regularity. We denote by $\sigma(\mathcal A)$ and $\rho(\mathcal A)$ the spectrum and the resolvent set of $\mathcal A$, respectively. Note that due to Theorem~\ref{!}, the operator $\mathcal A$ is closed and densely defined.

\begin{propo}\label{o in resolvet}
For all $\gamma, m,\rho\ge 0$, we have  $0\in\rho(\mathcal A)$.
\end{propo}

\begin{proof}
We  show that $\mathcal A\colon D(\mathcal A)\to\h$ is bijective. Let $f\in \h$. Then $\mathcal A w = f$ is equivalent to
\begin{equation}\label{3-1}
\langle \mathbb Aw,\phi\rangle_{\mathcal H_\gamma'\times \h }
=\langle\mathbb Mf ,\phi\rangle_{\mathcal H_\gamma'\times \h } \quad (\phi\in\h).
\end{equation}
Choosing $\phi=(\phi_1,0,0,0,0)^\top$ and $\phi=(0,0,0,\phi_4,0)^\top$, we obtain $\beta_1\Delta^2 w_2 = \beta_1\Delta^2 f_1$ and $-\beta_2\Delta w_5=-\beta_2\Delta f_4$, respectively, which has the unique solution
$w_2\coloneqq  f_1$ and $w_5\coloneqq  f_4$ (see \eqref{2-10}). Now choosing $\phi=(0,0,\phi_3,0,0)$, we obtain
\begin{equation}\label{3-2}
\beta_0\langle\nabla w_3,\nabla\phi_3\rangle_{L^2(\Omega_1)} +\beta_0 \kappa \langle w_3,\phi_3\rangle_{L^2(\Gamma)} = \langle \mu\Delta f_1-\rho_0 f_3 , \phi_3\rangle_{L^2(\Omega_1)}
\end{equation}
for all $\phi_3\in H^1(\Omega_1)$ with $\phi_3=0$ on $I$. As $\mu\Delta f_1-\rho_0f_3\in L^2(\Omega_1)$, the right-hand side is a continuous conjugate linear functional of $\phi_3$. Let us denote the left-hand side of \eqref{3-2} by $b(w_3,\phi_3)$. Then $b$ is a continuous sesquilinear form in the Hilbert space $\{ w_3\in H^1(\Omega_1): w_3=0\text{ on } I\}$.  From Remark~\ref{2.2} a) we know that the left-hand side is equivalent to the $H^1(\Omega_1)$-norm, which shows that $b(\cdot,\cdot)$ is coercive. Now an application of the theorem of Lax--Milgram yields the existence of a unique solution~$w_3$ of \eqref{3-2}.

For the remaining components $w_1$ and $w_4$, we choose $\phi=(0,\phi_2,0,0,\phi_5)^\top$ in~\eqref{3-1} and obtain
\begin{equation}\label{3-3}
\begin{aligned}
-\big\langle (w_1,w_4), &(\phi_2,\phi_5)\big\rangle_{H^{2,1}} =  \mu \langle w_3,\Delta\phi_2\rangle_{L^2(\Omega_1)} \\
& + \rho\langle\nabla w_2,\nabla\phi_2\rangle_{L^2(\Omega_1)} +  m\langle w_5,\phi_5\rangle_{L^2(\Omega_2)}\\
& +\rho_1\langle f_2, \phi_2\rangle_{L^2(\Omega_1)}+\gamma\langle\nabla f_2, \nabla\phi_2\rangle_{L^2(\Omega_1)} + \rho_2\left\langle f_5, \phi_5\right\rangle_{L^2(\Omega_2)}\\
& \eqqcolon R(\phi_2,\phi_5).
\end{aligned}
\end{equation}
Because of $(\phi_2,\phi_5)\in H^{2,1}$, the conjugate linear functional $R\colon H^{2,1}\to \C$ is well-defined and continuous. By the theorem of Riesz, there exists a unique solution $(w_1,w_4)\in H^{2,1}$ of \eqref{3-3}. Setting $w\coloneqq (w_1,\ldots,w_5)^\top$, we obtain $w\in \mathcal V$ (note here that $(w_2,w_5)^\top=(f_1,f_4)^\top\in H^{2,1}$), and $w$ is a solution of \eqref{3-1}. In particular, $\mathbb Aw\in \mathcal H_\gamma'$ by construction, so we have $w\in D(\mathcal A)$, and $\mathcal A$ is surjective. As the solution $w$ constructed above is unique, we also obtain the injectivity of $\mathcal A$.
As $\mathcal A\colon D(\mathcal A)\to\h$ is bijective and closed, we get $0\in\rho(\mathcal A)$.
\end{proof}

For the proof of higher regularity of the solution $w$, we need a priori-estimates from the theory of  parameter-elliptic boundary value problems as developed, e.g., in  \cite{AGRANOVICH1964ELLIPTIC}. We recall the main definitions and results (see \cite{Agranovich15}, Section 7.1). Let $(A(D),B_1(D),\dots, B_m(D))$ be a boundary value problem in some domain $\Omega\subset\R^n$ with $A(D) =   \sum_{|\alpha|\leq2m}a_{\alpha}\partial^{\alpha}$ and $B_{j}(D)= \sum_{|\beta|\leq m_{j}}b_{j\beta}\partial^{\beta}$, where $a_\alpha, b_{j\beta}\in\C$ and $m_j<2m$. Then the principal symbols of $A$ and $B_j$ are defined by $A(i\xi)\coloneqq \sum_{|\alpha|=2m}a_{\alpha}(i\xi)^{\alpha}$ and $B_{j}(i\xi)\coloneqq \sum_{|\beta|=m_{j}}b_{j\beta}(i\xi)^{\beta}$, respectively. The operator $A(D)$ is called  parameter-elliptic if its principal symbol satisfies
$$
\lambda-A(i\xi)\not =0\quad(\operatorname{Re}\lambda\geq0,\,\xi\in
\mathbb{R}^{n},\,(\lambda,\xi)\not =0).
$$
The boundary value problem is $(A,B_1,\dots,B_m)$ is called  parameter-elliptic if $A(D)$ is parameter-elliptic and if the following Shapiro--Lopatinskii condition holds:

Let $x_{0}\in\partial\Omega$, and rewrite the boundary value problem
in the coordinate system associated with $x_{0}$, which is obtained from the original one by a rotation after which the positive $x_{n}$-axis has the direction of the interior normal vector to $\partial\Omega$ at $x_{0}$. Then the trivial solution $w=0$ is the only stable solution of the ordinary differential equation on the half-line
\begin{align*}
\left(  \lambda-A(i\xi',\partial_{n})\right)  w(x_{n})  &  =0\quad(x_{n}
\in(0,\infty)),\\
B_{j}(i\xi',\partial_{n})w(0)  &  =0\quad(j=1,\dots,m)
\end{align*}
for all $\xi'\in\mathbb{R}^{n-1}$ and $\Re\lambda\geq0$ with $(\xi',\lambda)\not =0$.

In \cite{AGRANOVICH1964ELLIPTIC}, Theorem~5.1, the following result was shown:

\begin{teor}
  \label{Cor Regu boundaryProbNOnHom}
  Let $(A,B_1,\dots,B_m)$ be parameter-elliptic in $\Omega$. Then for sufficiently large $\lambda_{0}>0$, the boundary
	value problem
	\begin{equation*}
	\begin{array}
	[c]{rl}%
	\left(  \lambda_{0}-A(D)\right)  u=f & \text{in } \ \Omega \text{,}\\
	B_{j}(D)u=g_{j} & \text{on } \ \partial\Omega\text{,  }j=1,...,m\text{, }\\
	\end{array}
	\end{equation*}
	has a unique solution $u\in H^{2m}(\Omega)$, and the a priori-estimate
\begin{align*}
\|u\|_{H^{2m}(\Omega)} & \le C \Big( \|f\|_{L^2(\Omega) }+  \sum_{j=1}^m \|g_j\|_{H^{2m-m_j-1/2}(\partial\Omega)}\Big)
\end{align*}
	holds with a constant $C>0$ which depends on $\lambda_{0}$ but not on $u$ or
	on the data.
\end{teor}

\begin{obser}
  \label{3.1}
  a) We will apply this also in the case $\Omega=\Omega_1$, where $\partial\Omega_1= I\cup\Gamma$. It was shown in \cite{martinez2019regularity}, Remark~4.4, that we may also consider different boundary operators (even with different orders) in $I$ and $\Gamma$, respectively. One obtains unique solvability and the above a priori-estimate, where now the boundary norm for $g_j$ is given as the sum $\|g_j\|_{H^{2m-m_j'-1/2}(I)}+ \|g_j\|_{H^{2m-m_j''-1/2}(\Gamma)}$ with $m_j'$ and $m_j''$ being the order of $B_j$ on $I$ and $\Gamma$, respectively.

  b) It is well known (see, e.g., \cite{Agranovich15}, Subsection~7.1) that the Laplace operator is parameter-elliptic with Dirichlet boundary condition and with Neumann boundary condition. As only the principal part is involved in the definition of parameter-ellipticity, also $\Delta$ with mixed boundary condition $\partial_{\nu}u+\kappa u=0$ is parameter-elliptic. The same holds for $-\Delta^2$ with boundary conditions $u=\partial_\nu u=0$ (\cite{Agranovich15}, Remark~7.1.2).
\end{obser}

\begin{lem}\label{Lem ParameterEllipOps}
The operator $A(D) \coloneqq  -\Delta^{2}$ in $\Omega_{1}$, supplemented with the boundary operators $B_1(D) u \coloneqq  \partial_\nu u$ and $B_2(D) u \coloneqq  \partial_\nu \Delta u$, is parameter-elliptic.
\end{lem}

\begin{proof}
	Let $\lambda\in\C$, $\xi\in\R^2$ with $\Re\lambda\geq0$ and $(\lambda, \xi)\neq0$. Because of
	$\lambda-A(i\xi)=\lambda+|\xi|^4\neq0$, the operator $-\Delta^2$  is parameter-elliptic. For the Shapiro--Lopatinskii condition, we have to solve the  ordinary differential equation
	\begin{align}
  (\lambda+(\partial_2^2-\xi_1^2)^2) w(x_2) & = 0 \quad (x_2>0),\label{3-5a}\\
  \partial_2 w(0) & = 0,\label{3-5b}\\
  \partial_2^3 w(0) - \xi_1^2\partial_2 w(0) & = 0.\label{3-5c}
\end{align}
Note that by \eqref{3-5b}, we can replace \eqref{3-5c}  by $\partial_2^3 w(0)=0$. Let $\tau_{1,2} = -\sqrt{\xi_1^2\pm \sqrt{-\lambda}}$ be the two roots of the polynomial $\lambda-A(i\xi_1,\cdot)$ with negative real part.
For $\lambda\not=0$, we have $\tau_1\not=\tau_2$, and therefore every stable solution of \eqref{3-5a} has the form $w(x_2) = c_1 e^{\tau_1 x_2} + c_2 e^{\tau_2 x_2}$. Inserting this into the initial conditions, we obtain
\[ \begin{pmatrix}
  \tau_1 & \tau_2 \\ \tau_1 ^3 & \tau_2^3
\end{pmatrix}\binom{c_1}{c_2} =0.\]
As the determinant of this matrix equals $\tau_1\tau_2(\tau_2^2-\tau_1^2)\not=0$, we get $c_1=c_2=0$ and therefore $w=0$.

If $\lambda=0$, we have  $\tau_1 = \tau_2 = -|\xi_1|$, and $w(x_2) = (c_1+c_2x_2)e^{\tau_1 x_2}$. Now the initial conditions yield
\[ \begin{pmatrix}
  \tau_1 & 1 \\ \tau_1 ^3 & 3 \tau_1^2
\end{pmatrix}\binom{c_1}{c_2} =0,\]
which implies $w=0$ again.
\end{proof}

In the following, we will show that $D(\mathcal A)$ is embedded into a tuple of Sobolev spaces of higher regularity. For the continuity of the embedding, we use the following observation.

\begin{lem}\label{lem_continuity}
  Let $A\colon H\supset D(A)\to H$ be a closed operator in the Hilbert space $H$, and let $V$ be a Hilbert space. If  $D(A)$ is a subset of $V$, then we have the continuous embedding $D(A)\subset V$.
\end{lem}

\begin{proof}
  As $A$ is closed, $D(A)$ with the graph norm is a Hilbert space. We show that $\operatorname{id}\colon D(A)\to V\cap H$ is a closed operator. For this, let $(x_n)_{n\in\N}\subset D(A)$ be a sequence with $x_n\to x$ in $D(A)$ and $x_n\to y$ in $V\cap H$. Then we obtain $x_n\to x$ in $H$ by the definition of the graph norm, and also $x_n\to y$ in $H$ by the definition of the norm in $V\cap H$. This yields $x=y$, and $\operatorname{id}\colon D(A)\to V\cap H$ is closed and, by the closed graph theorem, continuous. As the embedding $V\cap H\to V$ is continuous by the definition of the norms, we obtain the continuity of $\operatorname{id}\colon D(A)\to V$.
\end{proof}

The elliptic regularity results above are the key for the strong solvability
of the transmission problem, that is, for higher regularity of the weak solution.

\begin{teor}\label{Th regularity}
	Let $\gamma, \rho, m\geq0$. Then the following embeddings are continuous.
	\begin{enumerate}[\upshape (i)]
		\item 	$ D(\mathcal{A})\;\,\subset H^2(\Omega_1) \times H^2(\Omega_1) \times H^2(\Omega_1) \times H^2(\Omega_2)\times H^1(\Omega_2) $,
		\item $ D(\mathcal{A}^{2})\subset H^4(\Omega_1) \times H^2(\Omega_1) \times H^2(\Omega_1) \times H^2(\Omega_2)\times H^1(\Omega_2) $,
		\item $ D(\mathcal{A})\subset H^4(\Omega_1) \times H^2(\Omega_1) \times H^2(\Omega_1) \times H^2(\Omega_2)\times H^1(\Omega_2) $ for $ \gamma=0 $.
	\end{enumerate}
	In consequence, if $w^0\in D(\mathcal{A}^{2})$ then $w(t)\coloneqq \mathcal{S}(t)w^0$ $(t\geq0)$ is the unique solution of problem \eqref{pl-1}-\eqref{inc} 
	and satisfies the boundary and transmission conditions in the strong sense of traces. In the case $ \gamma=0 $ we get the same result even for $ w^{0}\in D(\mathcal{A}) $.
\end{teor}
\begin{proof}
	(i)
	Let $w\in D(\mathcal{A})$ and $f\coloneqq \mathcal{A} w$.
 First, we show $ w_{3}\in H^{2}(\Omega_{1}) $.
	As in~\eqref{3-1}, we get
	\begin{align}\label{eq_Aw=f}
	\langle \mathbb Aw,\phi\rangle_{\mathcal H_\gamma'\times \h }=\langle\mathbb Mf ,\phi\rangle_{\mathcal H_\gamma'\times \h }
		\quad (\phi\in\h).
	\end{align}
	 As we already have seen in~\eqref{3-2}, for $ \phi_{3}\in H^{1}(\Omega_{1}) $ with $ \phi_{3}=0 $ on $ I $, we have
	\begin{align*}
		\left\langle g, \phi_{3}\right\rangle_{L^2(\Omega_1)}
		=\left\langle \nabla w_{3}, \nabla\phi_{3}\right\rangle_{L^2(\Omega_1)}+\left\langle \kappa w_{3}, \phi_{3}\right\rangle_{L^2(\Gamma)},
	\end{align*}
	where $ g\coloneqq \frac{1}{\beta_0}\left(\mu\Delta w_{2}-\rho_0 f_{3}\right)$.
		By Theorem~\ref{Cor Regu boundaryProbNOnHom} and Remark~\ref{3.1}, there exists some $\lambda_0>0$ such that the problem
	\begin{alignat*}{2}
	\left(\lambda_0-\Delta\right)\widetilde{w_{3}}=&\lambda_0 w_{3}+g \ && \textup{ in } \ \Omega_{1},\\
	\partial_{\nu}\widetilde{w_{3}}=&-\kappa w_{3} \ &&\textup{ on } \ \Gamma,\\
	\widetilde{w_{3}}=&0 \ &&\textup{ on } \ I
	\end{alignat*}
	has a unique solution $\widetilde{w_{3}}\in H^2(\Omega_1)$.
	Integration by parts shows that $z\coloneqq \widetilde{w_{3}}-w_{3}$ satisfies
	\begin{align*}
		0=&\langle (\lambda_{0}-\Delta)\widetilde{w_{3}} - \lambda_{0}w_{3}-g,\phi_{3}\rangle_{L^{2}(\Omega_{1})}
		= \lambda_0\left\langle z, \phi_{3}\right\rangle_{L^2(\Omega_1)}+\left\langle \nabla z, \nabla\phi_{3}\right\rangle_{L^2(\Omega_1)}
	\end{align*}
	for all $\phi_{3}\in H^{1}(\Omega_{1})$ with $ \phi_{3}=0 $ on $ I $. Choosing $\phi_{3}=z$, we get $w_{3}=\widetilde{w_{3}}\in H^2(\Omega_1)$.
	
	Now, we prove $w_{4}\in H^2(\Omega_2)$. We choose $ \phi=(0,0,0,0,\phi_{5})$ with $ \phi_{5}\in H^{1}_{0}(\Omega_{2}) $ in~\eqref{eq_Aw=f}. As in~\eqref{3-3}, we obtain
	\begin{equation*}\label{omega2}
	\langle\nabla w_{4}, \nabla \phi_{5}\rangle_{L^2 (\Omega_{2})} = \langle\widetilde{g},\phi_{5}\rangle_{L^2 (\Omega_{2})} ,
	\end{equation*}
	where $ \widetilde{g}\coloneqq -\frac{1}{\beta_{2}}(mw_{5}+\rho_{2}f_{5}) $.
	By Theorem~\ref{Cor Regu boundaryProbNOnHom} and Remark~\ref{3.1}b), there exists a unique $\widetilde{w_{4}}\in H^2(\Omega_2)$ such that
	\begin{alignat*}{2}
		-\Delta\widetilde{w_{4}}=&\widetilde{g} \ && \textup{ in } \ \Omega_2\\
		\widetilde{w_{4}} =& w_{1} \ && \textup{ on } \ I.
	\end{alignat*}
	Therefore $z\coloneqq \widetilde{w_{4}}-w_{4}\in H_{0}^{1}(\Omega_{2})$ fulfils
	\begin{align*}
		0 = \langle-\Delta \widetilde{w_{4}} - \widetilde{g}, \phi_{5}\rangle_{L^2 (\Omega_{2})}
		= \langle\nabla z,\nabla \phi_{5}\rangle_{L^2 (\Omega_{2})}
	\end{align*}
	for all $ \phi_{5}\in H_{0}^{1}(\Omega_{2}) $. By choosing $ \phi_{5}=z $, we  obtain $w_{4}=\widetilde{w_{4}}\in H^2(\Omega_2)$.

	(ii)
	Now, let $ w\in D(\mathcal{A}^{2}) $. We show $ w_{1}\in H^{4}(\Omega_{1}) $. In~\eqref{eq_Aw=f} we can choose $ \phi=(0,\phi_{2},0,0,0) $ for all $ \phi_{2}\in H^{2}(\Omega_{1}) $ with $ \phi_{2}=\partial_{\nu}\phi_{2}=0 $ on $ \Gamma $ and $ \partial_{\nu}\phi_{2}=0 $ on $ I $.  Integration by parts yields to
	\begin{align*}
		&\langle\Delta w_{1}, \Delta \phi_{2}\rangle_{L^2 (\Omega_{1})} \\
		=& \frac{1}{\beta_{1}} \left(  -\mu \langle w_{3},\Delta \phi_{2}\rangle_{L^2 (\Omega_{1})} -  \langle\nabla (\rho w_{2} + \gamma f_{2}), \nabla \phi_{2}\rangle_{L^2 (\Omega_{1})} - \rho_{1}\langle f_{2},\phi_{2}\rangle_{L^2 (\Omega_{1})}\right)\\
		=& \langle g^{*}, \phi_{2}\rangle_{L^2 (\Omega_{1})} - \langle h,\phi_{2}\rangle_{L^2 (I)},
	\end{align*}
	where $ g^{*}\coloneqq  \frac{1}{\beta_1}\left( \Delta (-\mu w_{3} + \rho w_{2}+\gamma f_{2}) -\rho_1f_{2}\right)$ and $ h\coloneqq \frac{1}{\beta_{1}}\partial_{\nu}(-\mu w_{3}+\rho w_{2}+\gamma f_{2}) $.
	By Theorem \ref{Cor Regu boundaryProbNOnHom} and Lemma~\ref{Lem ParameterEllipOps}, there is a $ \lambda_{0}>0 $ such that there exists a unique solution $\widetilde{w_1}\in H^4(\Omega_1)$ of the boundary value problem
	\begin{alignat*}{2}
	\left(\lambda_0+\Delta^2\right)\widetilde{w_1}=&\lambda_0 w_1+g^* \ &&\textup{ in } \ \Omega_1,\\
	\widetilde{w_1}=\partial_{\nu}\widetilde{w_1}=&0 \ &&\textup{ on } \ \Gamma,\\
	\partial_{\nu}\widetilde{w_1}=&0 \ &&\textup{ on } \ I,\\
	\partial_{\nu}(\Delta\widetilde{w_1})=&h \ && \textup{ on } \ I.
	\end{alignat*}
		Note that $ g^{*}\in L^{2}(\Omega_{1}) $ and $h\in H^{1}(\Omega_{1})$ since  $ w\in D(\mathcal{A}^{2}) $. Therefore, all boundary conditions hold in the trace sense.
	Using integration by parts, $ z\coloneqq \widetilde{w_{1}}-w_{1} $ fulfils
	\begin{align*}
	0=\big\langle (\lambda_0+\Delta^2)\widetilde{w_1}-\lambda_{0}w_{1}-g^{*}, \phi_2\big\rangle_{L^2(\Omega_1)}
	=\lambda_0\left\langle z, \phi_2\right\rangle_{L^2(\Omega_1)}+\left\langle \Delta z, \Delta\phi_2\right\rangle_{L^2(\Omega_1)}
	\end{align*}
	for all $ \phi_{2}\in H^{2}(\Omega_{1}) $ with $ \phi_{2}=\partial_{\nu}\phi_{2}=0 $ on $ \Gamma $ and $ \partial_{\nu}\phi_{2}=0 $ on $ I $. By choosing $ \phi_{2}=z $, we  obtain $w_{1}=\widetilde{w_{1}}\in H^4(\Omega_1)$.
	
	(iii) Let $ \gamma=0 $ and $ w\in D(\mathcal{A}) $. Following the proof of (ii), we get $ w_{1}\in H^{4}(\Omega_{1}) $.
	
	Due to Lemma~\ref{lem_continuity}, all embeddings are continuous.

\end{proof}

\begin{obser}
	\rm{
		By the last proof, we see that the corresponding assertions of Theorem~\ref{Th regularity} hold true if the plate is isothermal.
	}
\end{obser}

\begin{coro}\label{cor_regularity}
	For all $\gamma, \rho, m \geq0$, we have the continuous embedding
	\begin{equation*}\label{embedding}
	D(\mathcal{A})\subset H^3(\Omega_1) \times H^2(\Omega_1) \times H^2(\Omega_1) \times H^2(\Omega_2)\times H^1(\Omega_2).
	\end{equation*}
	\end{coro}

\begin{proof}
	From Theorem~\ref{!} we know that $ \mathcal{A}\colon \mathcal{H}_{\gamma}\supset D(\mathcal{A})\longrightarrow\mathcal{H}_{\gamma} $ is the generator of a $ C_{0} $-semigroup of contractions on $ \mathcal{H}_{\gamma} $. So, $-\mathcal{A}$ is an m-accretive operator (see Section 4.3. from \cite{lunardi2018interpolation}). By Corollary 4.30 and Corollary 4.37 from \cite{lunardi2018interpolation}, we obtain $D(\mathcal{A})=\left(\mathcal{H}_{\gamma}, D(\mathcal{A}^{2})\right)_{\frac{1}{2},2} $. Due to Theorem~\ref{Th regularity} it holds
	\begin{align*}
	D(\mathcal{A})\subset& H^2(\Omega_1) \times H^2(\Omega_1) \times H^2(\Omega_1) \times H^2(\Omega_2)\times H^1(\Omega_2),\\
	D(\mathcal{A}^{2})\subset& H^4(\Omega_1) \times H^2(\Omega_1) \times H^2(\Omega_1) \times H^2(\Omega_2)\times H^1(\Omega_2).
	\end{align*}
	By Proposition 5.12 from \cite{Bienvenido2019Vector-valued}, we have
	\begin{align*}
	D(\mathcal{A})&\subset
	(H^{2}(\Omega_{1}), H^{4}(\Omega_{1})_{\frac{1}{2},2}\times H^2(\Omega_1) \times H^2(\Omega_1) \times H^2(\Omega_2)\times H^1(\Omega_2).
	\end{align*}
	By Theorem 1 of Subsection 4.3.1. from \cite{Triebel78}, we get $\left(H^2(\Omega_1), H^4(\Omega_1)\right)_{\frac{1}{2}, 2}= H^{3}(\Omega_1)$.	
\end{proof}

The following result allows us to affirm that the spectrum $\sigma(\mathcal{A})$ of $\mathcal{A}$ coincides with its point spectrum $\sigma_p(\mathcal{A})$.

\begin{propo}
	The operator $\mathcal{A}^{-1} : \h\longrightarrow\h$ is compact.
\end{propo}
\begin{proof}
	By Corollary~\ref{cor_regularity} and the Rellich--Kondrachov theorem, we have
	\begin{align*}
	D(\mathcal{A})\subset& H^3(\Omega_1) \times H^2(\Omega_1) \times H^2(\Omega_1) \times H^2(\Omega_2)\times H^1(\Omega_2)\\
	\overset{c}{\subset}& H^2(\Omega_1) \times H^1(\Omega_1) \times L^2(\Omega_1) \times H^1(\Omega_2)\times L^2(\Omega_2).
	\end{align*}
	As $\h$ is a closed subspace of $H^2(\Omega_1) \times H^1(\Omega_1) \times L^2(\Omega_1) \times H^1(\Omega_2)\times L^2(\Omega_2)$, we get $ D(\mathcal{A})\overset{c}{\subset} \h $.
	Therefore the identity operator $\operatorname{id} \colon D(\mathcal{A})\longrightarrow\h$ is compact. Proposition \ref{o in resolvet} implies the continuity of the operator $\mathcal{A}^{-1} \colon \h\longrightarrow D(\mathcal{A})$. In consequence, $\mathcal{A}^{-1}=\operatorname{id}\circ\mathcal{A}^{-1} \colon \h\longrightarrow\h$ is a compact operator.
\end{proof}

\begin{propo}\label{strongly stable}
	If $\gamma, \rho\geq0$ and $m>0$, then $i\R\subset\rho(\mathcal{A})$.
\end{propo}
\begin{proof}
	Let us suppose $\gamma\geq0$ and $m>0$. Since $\mathcal{A}^{-1}$ is compact,  the spectrum of $\mathcal{A}$ consists of eigenvalues only. Thus, we have to establish that there are no purely imaginary eigenvalues.
	Let $0\neq\lambda\in\R$ and $ w\in D(\mathcal{A})$ with $\mathcal{A} w=i\lambda w$. By \eqref{2-11a}, we have
	\begin{align}\label{eq_Aw}
\langle \mathbb Aw,\phi\rangle_{\mathcal H_\gamma'\times \h }=	i\lambda\langle w ,\phi\rangle_{ \h }
	\quad (\phi\in\h).
	\end{align}
	Using Remark~\ref{2.2}, we see that $ i\lambda w_{1}=w_{2} $ and $ i\lambda w_{4}=w_{5} $.
	Choosing $ \phi=w $ in \eqref{eq_Aw}, we obtain
	\begin{align}\label{eq_imaginarypart}
	 \lambda\textup{Im}\langle \mathbb{A}w,w \rangle_{\mathcal H_\gamma'\times \h }=	\lambda^{2} \|w\|_{\h}^{2},
	\end{align}
	and
	\begin{align}\label{eq_realpart}
	\begin{split}
			0=&  \Re \langle \mathbb Aw,w\rangle_{\mathcal H_\gamma'\times \h}  \\
	=& -\rho\|\nabla w_2\|^2_{L^2(\Omega_1)}
	- \beta_0\|\nabla w_3\|_{L^2(\Omega_1)}^2 - \beta_0\kappa \|w_3\|_{L^2(\Gamma)}^2
	-m\| w_5\|^2_{L^2(\Omega_2)}
	\end{split}
	\end{align}
	as in the proof of Theorem~\ref{!}.
	For $\rho>0$, we get $w_{2}=w_{3}=w_{5}=0$ due to \eqref{eq_realpart} and Poincar\'e's inequality. We conclude $ w=0 $.

	If $\rho=0$, we have $w_{3}=w_{5}=0$ and therefore $ w_{4}=0 $, i.e., $w=(w_1,w_2,0,0,0)^\top$. Equality~\eqref{eq_imaginarypart} leads to
	\begin{align*}
	2 \beta_1\|\Delta w_{2}\|_{L^2(\Omega_1)}^{2} & = 	
	  \lambda\textup{Im} \langle \mathbb Aw,w\rangle_{\mathcal H_\gamma'\times\h}\\
	&=\beta_1\|\Delta w_{2}\|_{L^2(\Omega_1)}^{2}+
	\lambda^{2}\left(\rho_{1}\|w_{2}\|_{L^{2}(\Omega_{1})}^{2}
	+ \gamma\|\nabla w_{2}\|_{L^{2}(\Omega_{1})}^{2}\right),
	\end{align*}
	\textcolor{black}{therefore,  we have}
	\begin{equation}\label{i3}
	\beta_1\|\Delta w_{2}\|_{L^2(\Omega_1)}^{2}
		= \lambda^{2}\left(\rho_{1}\|w_{2}\|_{L^{2}(\Omega_{1})}^{2}
		+ \gamma\|\nabla w_{2}\|_{L^{2}(\Omega_{1})}^{2}
		\right).
	\end{equation}
	Now, we choose $ \phi=(0,0,\phi_{3},0,0)\in\h $ in \eqref{eq_Aw} and obtain
$ 	0=  \mu \langle \Delta w_2, \phi_3\rangle_{L^2(\Omega_1)} $
	for all $ \phi_{3}\in H^{1}(\Omega_{1}) $ with $ \phi_{3}=0 $ on $ I $. 	
	In consequence, we get $\Delta w_{2}=0$. From~\eqref{i3} it follows that $w_{2}=0$. Finally, for any case of $\rho$, we have shown that $w=0$.
\end{proof}

\begin{obser}\label{remark strongly stable}
	a) Proposition~\ref{strongly stable} also holds true if the plate is isothermal and $\rho>0$. In the case $\rho=0$, the above proof does not work in its present form in the isothermal case.\\
	b) We will see in the proof of Theorem \ref{teorcondition} that $i\R\subset\rho(\mathcal A)$ holds also for $m=0$.
\end{obser}

\section{Exponential stability in the case of damped membrane}\label{Exponential_stability}
In this section, we will prove the exponential stability of the solution of the system \eqref{pl-1}-\eqref{inc} when the membrane is damped and when $\rho>0$ or $\rho=\gamma=0$.




\begin{teor}\label{exp stab}
	If $m>0$ and $\rho>0$, then for all $\gamma\ge 0$
 the semigroup $\left(\mathcal{S} (t)\right)_{t\geq0}$ generated by $\mathcal A$ is exponentially stable, that is, there exist constants $C\geq1$ and $\delta>0$ such that $\left\|\mathcal{S}  (t)\right\|_{\mathcal{L}(\h)}\leq Ce^{-\delta t}$ for all $t\geq0$.
\end{teor}
\begin{proof}
	Let $\gamma\geq0$, $m, \rho>0$ and $\lambda\in \R$. For the proof, we use the characterization of exponential stability by Gearhart and Pr\"uss (see \cite{Pruss1984C_0semigroup}) which tells us that the semigroup is exponentially stable if  $ i\R\subset\rho(\mathcal{A}) $  and there is a constant $ C>0 $, which does not depend on $ \lambda\in \R $, such that
	\begin{align}\label{desigualdad}
		\left\|(i\lambda  -\mathcal{A})^{-1}\right\|_{\mathcal{L}(\h)}\leq C.
	\end{align}
As $i\R\subset\rho(\mathcal{A}) $ by Proposition~\ref{strongly stable}, we have to show \eqref{desigualdad}.
	To see this, let $w\in D(\mathcal{A})$, $\lambda\in\R$, and
	\begin{equation}\label{resolvent equation}
	(i\lambda  -\mathcal{A})w\eqqcolon f.
	\end{equation}
	To prove \eqref{desigualdad}, it is sufficient to establish that there is a constant $ C>0 $  such that for all $ \varepsilon>0 $ there exists $C_{\varepsilon}>0$ such that
	\begin{equation}\label{e7}
	\left\|w\right\|^2_{\h}\leq C_{\varepsilon}\left\|w\right\|_{\h}\left\|f\right\|_{\h}+ \varepsilon C\left\|w\right\|^2_{\h}. 
	\end{equation}
	Multiplying the resolvent equation \eqref{resolvent equation} by $w$, we obtain
	$$
	i\lambda\left\|w\right\|^2_{\h}-\left\langle \mathcal{A} w, w\right\rangle _{\h} =\left\langle f, w\right\rangle _{\h}.
	$$
	In consequence,
	$$
	-\Re\left\langle \mathcal{A} w, w\right\rangle _{\h}=\Re\left\langle f, w\right\rangle _{\h}\leq\big|\left\langle f, w\right\rangle _{\h}\big|.
	$$
	As we have seen in the proof of Theorem~\ref{!}, this means
	\begin{align}\label{dissip}
	\begin{split}
 \rho\|\nabla w_2\|^2_{L^2(\Omega_1)} &
	+ \beta_0\|\nabla w_3\|_{L^2(\Omega_1)}^2 + \beta_0\kappa \|w_3\|_{L^2(\Gamma)}^2
	+m\| w_5\|^2_{L^2(\Omega_2)}\\\leq&\left\|w\right\|_{\h}\left\|f\right\|_{\h}.
	\end{split}
	\end{align}
	From  Remark~\ref{2.2} and \eqref{resolvent equation} it follows that
	\begin{align}\label{equality1}
	\begin{split}
	(i\lambda w_1-w_{2}, i\lambda w_{4}-w_{5}) & =(f_1, f_{4}),\\
	\langle \mathbb{M}(i\lambda w - f), \phi \rangle_{\mathcal H_\gamma'\times \h } & =  \langle \mathbb Aw,\phi\rangle_{\mathcal H_\gamma'\times \h }.
	\end{split}
	\end{align}
	Choosing $ \phi=(0,w_{1},0,0,w_{4}) $, we get
	\begin{align*}
	&\rho_{1}\langle  i\lambda w_{2}-f_{2}, w_{1}\rangle_{L^2 (\Omega_{1})} + \gamma \langle \nabla (i\lambda w_{2}-f_{2}), \nabla w_{1}\rangle_{L^2 (\Omega_{1})} + \rho_{2}\langle i\lambda w_{5} - f_{5}, w_{4}\rangle_{L^2 (\Omega_{2})}\\
	=&-\|(w_1,w_4)\|_{H^{2,1}}^{2}
	 - \mu \langle w_3,\Delta w_{1}\rangle_{L^2(\Omega_1)}  - \rho\langle\nabla w_2,\nabla w_{1}\rangle_{L^2(\Omega_1)}  - m\langle w_5, w_{4}\rangle_{L^2(\Omega_2)}.
	\end{align*}
	Taking \eqref{equality1} into account, we get
	\begin{align*}
	&\rho_{1}\langle i\lambda w_{2}, w_{1}\rangle_{L^2 (\Omega_{1})} + \gamma \langle i\lambda\nabla w_{2}, \nabla w_{1}\rangle_{L^2 (\Omega_{1})} + \rho_{2}\langle i\lambda w_{5}, w_{4}\rangle_{L^2 (\Omega_{2})}\\
	=&-\rho_{1}\langle  w_{2},w_{2}+f_{1}\rangle_{L^2 (\Omega_{1})} - \gamma \langle \nabla  w_{2}, \nabla (w_{2}+f_{1})\rangle_{L^2 (\Omega_{1})} - \rho_{2}\langle w_{5},w_{5}+f_{4}\rangle_{L^2 (\Omega_{2})}\\
	=&-\rho_{1}\|w_{2}\|_{L^{2}(\Omega_{1})}^{2} - \gamma \|\nabla  w_{2}\|_{L^{2}(\Omega_{1})}^{2} - \rho_{2}\| w_{5}\|_{L^{2}(\Omega_{2})}^{2}\\
	&-\rho_{1}\langle  w_{2},f_{1}\rangle_{L^2 (\Omega_{1})} - \gamma \langle \nabla  w_{2}, \nabla f_{1}\rangle_{L^2 (\Omega_{1})} - \rho_{2}\langle  w_{5},f_{4}\rangle_{L^2 (\Omega_{2})}.
	\end{align*}
	Using the last equality, Poincar\'{e}'s inequality and inequality \eqref{dissip}, we obtain
	\begin{align*}
		\left\|(w_1, w_{4})\right\|^2_{H^{2,1}}=&\rho_{1}\|w_{2}\|_{L^{2}(\Omega_{1})}^{2} + \gamma \|\nabla  w_{2}\|_{L^{2}(\Omega_{1})}^{2} + \rho_{2}\| w_{5}\|_{L^{2}(\Omega_{2})}^{2}\\
		&+\rho_{1}\langle  w_{2},f_{1}\rangle_{L^2 (\Omega_{1})} + \gamma \langle \nabla  w_{2}, \nabla f_{1}\rangle_{L^2 (\Omega_{1})} + \rho_{2}\langle w_{5},f_{4}\rangle_{L^2 (\Omega_{2})}\\
		&+ \rho_{1}\langle f_{2}, w_{1}\rangle_{L^2 (\Omega_{1})}+ \gamma \langle \nabla f_{2}, \nabla w_{1}\rangle_{L^2 (\Omega_{1})} + \rho_{2}\langle f_{5}, w_{4}\rangle_{L^2 (\Omega_{2})}\\
		&-\mu\left\langle w_{3}, \Delta w_1\right\rangle_{L^2(\Omega_1)}
		-\rho\left\langle \nabla w_{2}, \nabla w_1\right\rangle_{L^2(\Omega_1)}-m\left\langle w_{5}, w_{4}\right\rangle_{L^2(\Omega_2)}\\
		\leq&   C\left\|w\right\|_{\h}\left\|f\right\|_{\h}
		+\mu\left|\left\langle w_{3}, \Delta w_1\right\rangle_{L^2(\Omega_1)}\right|\\
		&+\rho\left|\left\langle \nabla w_{2}, \nabla w_1\right\rangle_{L^2(\Omega_1)}\right|+m\left|\left\langle w_{5}, w_{4}\right\rangle_{L^2(\Omega_2)}\right|.
	\end{align*}	
	Now, we estimate the remaining terms on the right-hand side. Due to Young's and Poincar\'{e}'s inequality, we have
	\begin{align*}
\left|\left\langle w_{3}, \Delta w_1\right\rangle_{L^2(\Omega_1)}\right|\leq& C_{\varepsilon}\left\|w_{3}\right\|^2_{L^2(\Omega_1)}+\varepsilon\left\|\Delta w_1\right\|^2_{L^2(\Omega_1)}\\
\leq & C_{\epsilon} \left\|\nabla w_{3}\right\|^2_{L^2(\Omega_1)} + \epsilon C \|w\|_{\h}^{2}\\
\leq & C_{\epsilon} \left\|w\right\|_{\h}\left\|f\right\|_{\h} + \epsilon C \|w\|_{\h}^{2}
\end{align*}
for all $\varepsilon>0$. In the last step, we used again inequality~\eqref{dissip}. Similarly, we get
\begin{align*}
	\left|\left\langle \nabla w_{2}, \nabla w_1\right\rangle_{L^2(\Omega_1)}\right|
	&\leq C_{\varepsilon}\left\|w\right\|_{\h}\left\|f\right\|_{\h}+\varepsilon C\left\|w\right\|^2_{\h}
	\intertext{and}
	\left|\left\langle w_{5}, w_{4}\right\rangle_{L^2(\Omega_2)}\right|
	&\leq C_{\varepsilon}\left\|w\right\|_{\h}\left\|f\right\|_{\h}+\varepsilon C\left\|w\right\|^2_{\h}
\end{align*}
for all $ \epsilon>0 $. Altogether, we have shown inequality~\eqref{e7}. Therefore, the semigroup is exponentially stable.
\end{proof}

\begin{obser}
		Our plate-membrane system also has exponential stability when the plate is isothermal.
\end{obser}

We will now show that the thermoelastic plate-membrane system without rotational inertia has exponential stability if the membrane is damped ($m>0$) even without structural damping ($\rho=0$). Under this situation, the thermal effect on the plate is enough for exponential decay. For the proof of this result, the following lemma (Theorem 1.4.4 in \cite{liu1999semigroups}) will be useful.

 \begin{lem}\label{Lemma_01}
 Let $\Omega\subset\R^n$ be a bounded domain with $C^1$-boundary. Then, for any function $u\in H^1(\Omega)$  the following estimate holds:
  \begin{equation*}\label{Ineq_01}
   \Vert u\Vert_{L^2(\partial\Omega)} \leq C \Vert u\Vert_{H^1(\Omega)}^{1/2}\Vert u\Vert_{L^2(\Omega)}^{1/2}
  \end{equation*}
\end{lem}

Below, we will also apply this lemma to $\partial_\nu u$, $\Delta u$, and $\partial_\nu\Delta u$ with  $u$ being sufficiently smooth.

\begin{teor}\label{Main_Th_section8}
If $\gamma=0$, $\rho=0$ and $m>0$, then the semigroup $\left(\mathcal{S}(t)\right)_{t\geq0}$ generated by the operator $\mathcal A$ is exponentially stable.
\end{teor}

\begin{proof}
Again we use the Gearhart--Pr\"uss criterion, so we study  the resolvent in $\mathcal H_0$ (note $\gamma=0$) on the imaginary axis.  From Proposition \ref{strongly stable}, we have $i\R\subset\rho(\mathcal A )$.
Let us suppose \eqref{desigualdad} is not true. Then, there exists a sequence $\left(\lambda_n\right)_{n\in\N}\subset\R$ and a sequence $\left(w^n\right)_{n\in\N}\subset D(\mathcal A )$ with $\left\|w^n\right\|_{\mathcal{H}_0}=1$ such that
\begin{equation}\label{ec2}
\left\|(i\lambda_n -\mathcal A )w^n\right\|_{\mathcal{H}_0}\to 0 \quad (n\to\infty).
\end{equation}
As the resolvent is holomorphic and therefore bounded on compact subsets of the imaginary axis, we see that the sequence $(\lambda_n)_{n\in\N}$ is unbounded, so we may assume $\lim_{n\rightarrow\infty}|\lambda_n|=\infty$.
For  $f^n := (i\lambda_n -\mathcal A )w^n$, we obtain
\begin{align}
i\lambda_nw_1^n-w_2^n&=f_1^n,\label{ec3}\\
i\lambda_n\rho_1w_2^n+\beta_1\Delta^2w_1^n+\mu\Delta w_3^n&=\rho_1f_2^n,\label{ec5}\\
i\lambda_n\rho_0w_3^n-\mu\Delta w_2^n-\beta_0\Delta w_3^n&=\rho_0f_3^n,\label{ec7}\\
i\lambda_nw_4^n-w_5^n&=f_4^n,\label{ec4}\\
i\lambda_n\rho_2w_5^n-\beta_2\Delta w_4^n+mw_5^n&=\rho_2f_5^n,\label{ec6}
\end{align}
and since
\begin{align}\label{ec8}
\begin{split}
\left\|(i\lambda_n -\mathcal A )w^n\right\|_{\mathcal{H}_0}^2&=\left\|f^n\right\|^2_{\mathcal{H}_0}=\beta_1\left\|\Delta f_1^n\right\|^2_{L^2(\Omega_1)} +\beta_2\left\|\nabla f_4^n\right\|^2_{L^2(\Omega_2) }\\
&+\rho_1\left\|f_2^n\right\|^2_{L^2(\Omega_1)}+\rho_2 \left\|f_5^n\right\|^2_{L^2(\Omega_2)}+\rho_0\left\|f_3^n\right\|^2_{L^2(\Omega_1)},
\end{split}
\end{align}
we obtain from \eqref{ec2}-\eqref{ec8}
\begin{align}
i\lambda_n\Delta w_1^n-\Delta w_2^n&\to 0 \ \textup{ in } \ L^2(\Omega_1),\label{ec9}\\
i\lambda_n\rho_1w_2^n+\beta_1\Delta^2w_1^n+\mu\Delta w_3^n& \to 0 \ \textup{ in } \ L^2(\Omega_1),\label{ec11}\\
i\lambda_n\rho_0w_3^n-\mu\Delta w_2^n-\beta_0\Delta w_3^n&\to 0  \ \textup{ in } \ L^2(\Omega_1),\label{ec13}\\
i\lambda_n\nabla w_4^n-\nabla w_5^n&\to 0  \ \textup{ in } \ L^2(\Omega_2),\label{ec10}\\
i\lambda_n\rho_2w_5^n-\beta_2\Delta w_4^n+mw_5^n&\to 0  \ \textup{ in } \ L^2(\Omega_2) \label{ec12}
\end{align}
for $n\to\infty$. From \eqref{2-11} it follows that
\begin{align*}
\Re\left\langle (i\lambda_n -\mathcal A )w^n, w^n\right\rangle _{\mathcal{H}_0}&=\Re\left[i\lambda_n\left\|w^n\right\|^2_{\mathcal{H}_0}-\left\langle \mathcal A  w^n, w^n\right\rangle _{\mathcal{H}_0}\right]\\
&=-\Re\left\langle \mathcal A  w^n, w^n\right\rangle _{\mathcal{H}_0}\\
&=m\left\|w_5^n\right\|^2_{L^2(\Omega_2)}+\beta_0\|\nabla w_3^n\|_{L^2(\Omega_1)}^2 + \beta_0\kappa \|w_3^n\|_{L^2(\Gamma)}^2.
\end{align*}
As $(i\lambda_n -\mathcal A )w^n$ converges to zero in $\mathcal{H}_0$ and $\left(w^n\right)_{n\in\N}$ is a bounded sequence in $\mathcal{H}_0$, the right-hand side of the last equality tends to zero. Therefore,
\begin{equation}\label{ec14}
 w_5^n\to 0\;\text{ in } L^2(\Omega_2)\quad\text{ and } w_3^n\to 0\;\text{ in } H^1(\Omega_1).
\end{equation}
Note also that the sequences  $(w_1^n)_{n\in\N}\subset H^2(\Omega_1)$, $(w_2^n)_{n\in\N}\subset L^2(\Omega_1)$, and $(w_4^n)_{n\in\N}\subset H^1(\Omega_1)$ are bounded because of $\|w^n\|_{\mathcal H_0}=1$. As $f^n\to 0$ in $\mathcal H_0$, we obtain $f_1^n\to 0$ in $H^2(\Omega_1)$ and $f_4^n\to 0$ in $H^1(\Omega_2)$ which yields
$$
\pm iw_1^n-\frac{w_2^n}{|\lambda_n|}\to 0 \ \textup{ in } \ H^2(\Omega_1) \ \textup{ and } \ \pm i w_4^n-\frac{w_5^n}{|\lambda_n|}\to 0 \ \textup{ in } \ H^1(\Omega_2).
$$
As $\left\|w_1^n\right\|_{H^2(\Omega_1)}\leq C $ and $\left\|w_4^n\right\|_{H^1(\Omega_2)}\leq C $, the sequence $\left(\frac{w_2^n}{|\lambda_n|}\right)_{n\in\N}$ is bounded in $H^2(\Omega_1)$, and the sequence $\left(\frac{w_5^n}{|\lambda_n|}\right)_{n\in\N}$ is bounded in $H^1(\Omega_2)$.

The convergences \eqref{ec12}, \eqref{ec13} and \eqref{ec14} imply
\begin{equation}\label{ec15}
\frac{1}{|\lambda_n|}\Delta w_4^n\to 0  \ \textup{ in } \ L^2(\Omega_2)
\end{equation}
and
\begin{equation}\label{ec16}
\frac{\mu}{|\lambda_n|}\Delta w_2^n+\frac{\beta_0}{|\lambda_n|}\Delta w_3^n\to 0  \ \textup{ in } \ L^2(\Omega_1).
\end{equation}
As $\left(\nabla w_4^n\right)_{n\in\N}$ is a bounded sequence in $L^2(\Omega_2) $, we get from \eqref{ec10} that $$\left\langle i\lambda_n\nabla w_4^n-\nabla w_5^n, \nabla w_4^n\right\rangle_{L^2(\Omega_2) }\to 0.$$
Using integration by parts, we have
\begin{align}\label{ec17}
\begin{split}
&\left\langle i\lambda_n\nabla w_4^n-\nabla w_5^n, \nabla w_4^n\right\rangle_{L^2(\Omega_2)}\\
&=i\lambda_n\left\|\nabla w_4^n\right\|^2_{L^2(\Omega_2)}+\left\langle w_5^n, \Delta w_4^n\right\rangle_{L^2(\Omega_2)}+\left\langle w_5^n, \partial_{\nu}w_4^n\right\rangle_{L^2(I)}.
\end{split}
\end{align}
With the interpolation inequality, we obtain
\begin{align*}
\left\|w_2^n\right\|_{H^1(\Omega_1)}&\leq C \left\|w_2^n\right\|^{1/2}_{H^2(\Omega_1)}\left\|w_2^n\right\|^{1/2}_{L^2(\Omega_1)}\\
&=C \left\|i\lambda_nw_1^n-f_1^n\right\|^{1/2}_{H^2(\Omega_1)}\left\|w_2^n\right\|^{1/2}_{L^2(\Omega_1)}\\
&\leq C \left(|\lambda_n|^{1/2}\left\|w_1^n\right\|^{1/2}_{H^2(\Omega_1)} +\left\|f_1^n\right\|^{1/2}_{H^2(\Omega_1)}\right)\left\|w_2^n\right\|^{1/2}_{L^2(\Omega_1)}
\end{align*}
and thus
\begin{equation}\label{tion1}
\frac{\left\|w_2^n\right\|_{H^1(\Omega_1)}}{|\lambda_n|^{1/2}}\leq C \Big(\left\|w_1^n\right\|^{1/2}_{H^2(\Omega_1)} +\frac{\left\|f_1^n\right\|^{1/2}_{H^2(\Omega_1)}}{|\lambda_n|^{1/2}}\Big)\left\|w_2^n\right\|^{1/2}_{L^2(\Omega_1)}\leq C .
\end{equation}
By trace theorem, Lemma \ref{Lemma_01} \textcolor{black}{applied to $\partial_\nu w_4^n$} and $w_5^n=w_2^n$ on $I$, we get
\begin{align*}
\left|\left\langle w_5^n, \partial_{\nu}w_4^n\right\rangle_{L^2(I)}\right|&\leq\left\|w_5^n\right\|_{L^2(I)}\left\|\partial_{\nu}w_4^n\right\|_{L^2(I)}\nonumber\\
&\leq C \left\|w_2^n\right\|_{H^1(\Omega_1)}\left\|w_4^n\right\|^{1/2}_{H^2(\Omega_2)}\left\|w_4^n\right\|^{1/2}_{H^1(\Omega_2)}
\end{align*}
and therefore
\begin{align}\label{ec18}
\left|\frac{1}{\lambda_n}\left\langle w_5^n, \partial_{\nu}w_4^n\right\rangle_{L^2(I)}\right|&\leq C \frac{\left\|w_2^n\right\|_{H^1(\Omega_1)}}{|\lambda_n|^{1/2}} \frac{\left\|w_4^n\right\|^{1/2}_{H^2(\Omega_2)}}{|\lambda_n|^{1/2}}\left\|w_4^n \right\|^{1/2}_{H^1(\Omega_2)}\nonumber\\
&\leq C \frac{\left\|w_4^n\right\|^{1/2}_{H^2(\Omega_2)}}{|\lambda_n|^{1/2}}.
\end{align}
By \eqref{ec6} we have
$$
\Delta w_4^n=\frac{1}{\beta_2}(i\lambda_n\rho_2w_5^n+mw_5^n-\rho_2f_5^n)
$$
and as $w_4^n=w_1^n$ on $I$,  elliptic regularity for the Dirichlet Laplacian
(Theorem~\ref{Cor Regu boundaryProbNOnHom})  implies
\begin{align*}\label{ec22}
\left\|w_4^n\right\|_{H^2(\Omega_2)}&\leq C \left(\left\|i\lambda_n\rho_2w_5^n+mw_5^n-\rho_2f_5^n \right\|_{L^2(\Omega_2)}+\left\|w_1^n\right\|_{H^{3/2}(I)}\right)\nonumber\\
&\leq C \left(|\lambda_n|\left\|w_5^n\right\|_{L^2(\Omega_2)}+ \left\|w_5^n\right\|_{L^2(\Omega_2)}+\left\|f_5^n\right\|_{L^2(\Omega_2)}
+\left\|w_1^n\right\|_{H^2(\Omega_1)}\right).
\end{align*}
Thus,
$$
\frac{\left\|w_4^n\right\|_{H^2(\Omega_2)}}{|\lambda_n|}\leq C \left(\left\|w_5^n\right\|_{L^2(\Omega_2)}+\frac{\left\|w_5^n \right\|_{L^2(\Omega_2)}}{|\lambda_n|}+\frac{\left\|f_5^n\right\|_{L^2(\Omega_2)}}{|\lambda_n|} +\frac{\left\|w_1^n\right\|_{H^2(\Omega_1)}}{|\lambda_n|}\right).
$$
Because of $\lim_{n\rightarrow\infty}\left\|w_5^n\right\|_{L^2(\Omega_2)}=0$, $\lim_{n\rightarrow\infty}\left\|f_5^n\right\|_{L^2(\Omega_2)}=0$ and $\left\|w_1^n\right\|_{H^2(\Omega_1)}\leq C$, we get
\begin{equation}\label{ec19}
\frac{\left\|w_4^n\right\|_{H^2(\Omega_2)}}{|\lambda_n|}\to 0 \quad (n\to\infty).
\end{equation}
From \eqref{ec15} and \eqref{ec17}, \eqref{ec18} and \eqref{ec19}, we obtain
\begin{equation}\label{ec20}
\nabla w_4^n\to 0  \ \textup{ in } \ L^2(\Omega_2).
\end{equation}

\noindent Dividing \eqref{ec9} by $|\lambda_n|$, we have
$$
\pm i\Delta w_1^n-\frac{1}{|\lambda_n|}\Delta w_2^n\to 0  \ \textup{ in } \ L^2(\Omega_1)
$$
and given that $\left\|\Delta w_1^n\right\|_{L^2(\Omega_1)}\leq C$, then $\left(\frac{1}{|\lambda_n|}\Delta w_2^n\right)_{n\in\N}$ is a bounded sequence in $L^2(\Omega_1)$. Consequently,
the limit \eqref{ec16} imply that $\left(\frac{1}{|\lambda_n|}\Delta w_3^n\right)_{n\in\N}$ is a bounded sequence in $L^2(\Omega_1)$. From \eqref{ec11} it follows that
$$
\pm i\rho_1w_2^n+\frac{\beta_1}{|\lambda_n|}\Delta^2w_1^n+\frac{\mu}{|\lambda_n|}\Delta w_3^n\to 0  \ \textup{ in } \ L^2(\Omega_1).
$$
Hence,
\begin{equation}\label{ec21}
\frac{1}{|\lambda_n|}\left\|\Delta^2w_1^n\right\|_{L^2(\Omega_1)}\leq C .
\end{equation}
\textcolor{black}{Due to \eqref{ec5} and $w^n\in D(\mathcal A )$ (see Theorem \ref{Th regularity}), $w_1^n$ satisfies the problem
$$
\begin{cases}
(\eta_0+\Delta^2)w_1^n=\eta_0w_1^n+\beta_1^{-1}z^n \ \textup{ in } \ \Omega_1,\\
w_1^n=0, \partial_{\nu}w_1^n=0 \ \textup{ on } \ \Gamma,\\
\partial_{\nu}w_1^n=0, \partial_{\nu}(\Delta w_1^n)=\beta_1^{-1}(-\beta_2\partial_{\nu}w_4^n-\partial_{\nu}w_3^n) \ \textup{ on } \ I,
\end{cases}
$$
with
$$
z^n:=\lambda^2_n\rho_1w_1^n-\mu\Delta w_3^n+i\lambda_n\rho_1f_1^n+\rho_1f_2^n.
$$
Then, Theorem \ref{Cor Regu boundaryProbNOnHom} implies
\begin{align}\label{ec23}
&\left\|w_1^n\right\|_{H^4(\Omega_1)}\nonumber\\
&\leq C \left(\left\|\eta_0w_1^n+\Delta^2w_1^n\right\|_{L^2(\Omega_1)}+ \left\|\beta_1^{-1}(-\beta_2\partial_{\nu}w_4^n-\mu\partial_{\nu}w_3^n)\right\|_{H^{1/2}(I)}\right)\nonumber\\
&\leq C \left(\left\|w_1^n\right\|_{L^2(\Omega_1)} +\left\|\Delta^2w_1^n\right\|_{L^2(\Omega_1)} +\left\|\partial_{\nu}w_4^n\right\|_{H^{1/2}(I)}+\left\|\partial_{\nu}w_3^n\right\|_{H^{1/2}(I)}\right).
\end{align}
}
By the trace theorem and \eqref{ec19}, we have
\begin{equation}\label{ec24}
\textcolor{black}{\frac{\left\|\partial_{\nu}w_4^n\right\|_{H^{1/2}(I)}}{|\lambda_n|}\leq C \frac{1}{|\lambda_n|}\left\|w_4^n\right\|_{H^2(\Omega_2)} \to 0}.
\end{equation}
Note that $w_3^n$ is a solution to the following problem
$$
\begin{cases}
\Delta w_3^n=\frac{\rho_0}{\beta_0}(i\lambda_nw_3^n-\rho_0^{-1}\mu\Delta w_2^n-f_3^n)=:h^*_n\in L^2(\Omega_1),\\
\partial_{\nu}w_3^n+\kappa w_3^n=0 \ \textup{ on } \ \Gamma,\\
w_3^n=0 \ \textup{ on } \ I.
\end{cases}
$$
In consequence,
\begin{align*}
\left\|w_3^n\right\|_{H^2(\Omega_1)}&\leq C\left\|\eta_0w_3^n-h^*_n\right\|_{L^2(\Omega_1)}\\
&\leq C \Big(\left\|w_3^n\right\|_{L^2(\Omega_1)}+|\lambda_n|\left\|w_3^n\right\|_{L^2(\Omega_1)}\\
&\quad+|\lambda_n|\left\|\Delta w_1^n\right\|_{L^2(\Omega_1)}+\left\|\Delta f_1^n\right\|_{L^2(\Omega_1)}+\left\|f_3^n\right\|_{L^2(\Omega_1)}\Big),
\end{align*}
here we have used equality \eqref{ec3}. Thus,
\begin{equation}\label{tion2}
\frac{1}{|\lambda_n|}\left\|w_3^n\right\|_{H^2(\Omega_1)}\leq C .
\end{equation}
Consequently,
\begin{equation}\label{ec25}
\frac{\left\|\partial_{\nu}w_3^n\right\|_{H^{1/2}(I)}}{|\lambda_n|}\leq C \frac{1}{|\lambda_n|}\left\|w_3^n\right\|_{H^2(\Omega_1)}\leq C .
\end{equation}
The estimates \eqref{ec21}, \eqref{ec23}, \eqref{ec24} and \eqref{ec25} imply $\frac{1}{|\lambda_n|}\left\|w_1^n\right\|_{H^4(\Omega_1)}\leq C $. By interpolation inequality
$$
\frac{1}{|\lambda_n|^{1/2}}\left\|w_1^n\right\|_{H^3(\Omega_1)}\leq C \frac{\left\|w_1^n\right\|^{1/2}_{H^4(\Omega_1)}}{|\lambda_n|^{1/2}}\left\|w_1^n\right\|^{1/2}_{H^2(\Omega_1)}.
$$
Therefore, $\left(\frac{1}{|\lambda_n|^{1/2}}w_1^n\right)_{n\in\N}$ is a bounded sequence in $H^3(\Omega_1)$. From \eqref{ec9} and \eqref{ec16} it follows that
$$
\pm i\mu\Delta w_1^n+\frac{\beta_0}{|\lambda_n|}\Delta w_3^n\to 0 \ \textup{ in } \ L^2(\Omega_1)
$$
and therefore
$$
\left\langle \pm i\mu\Delta w_1^n+\frac{\beta_0}{|\lambda_n|}\Delta w_3^n, \Delta w_1^n\right\rangle_{L^2(\Omega_1)}\to 0 .
$$
So,
\begin{equation}\label{ec26}
\pm i\mu\left\|\Delta w_1^n\right\|^2_{L^2(\Omega_1)}+\frac{\beta_0}{|\lambda_n|}\left(\Delta w_3^n, \Delta w_1^n\right\rangle_{L^2(\Omega_1)}\to 0 .
\end{equation}
Using integration by parts,
\begin{align}\label{ec27}
\begin{split}
&\left\langle \Delta w_3^n, \Delta w_1^n\right\rangle_{L^2(\Omega_1)}\\
&=\left\langle w_3^n, \Delta^2w_1^n\right\rangle_{L^2(\Omega_1)}-\left\langle w_3^n, \partial_{\nu}\Delta w_1^n\right\rangle_{L^2(\Gamma)}+\left\langle \partial_{\nu}w_3^n, \Delta w_1^n\right\rangle_{L^2(\partial\Omega_1)}.
\end{split}
\end{align}
By the trace theorem and Lemma \ref{Lemma_01} applied to $\partial_\nu\Delta w_1^n$, we have
\begin{align*}
\left|\left\langle w_3^n, \partial_{\nu}\Delta w_1^n\right\rangle_{L^2(\Gamma)}\right|&\leq\left\|w_3^n\right\|_{L^2(\Gamma)}\left\|\partial_{\nu}\Delta w_1^n\right\|_{L^2(\Gamma)}\\
&\leq C \left\|w_3^n\right\|_{H^1(\Omega_1)}\left\|w_1^n\right\|^{1/2}_{H^4(\Omega_1)}\left\|w_1^n\right\|^{1/2}_{H^3(\Omega_1)}.
\end{align*}
Then,
\begin{align*}
\left|\frac{1}{\lambda_n}\left\langle w_3^n, \partial_{\nu}\Delta w_1^n\right\rangle_{L^2(\Gamma)}\right|&\leq C \frac{\left\|w_3^n\right\|_{H^1(\Omega_1)}}{|\lambda_n|^{1/4}} \frac{\left\|w_1^n\right\|^{1/2}_{H^4(\Omega_1)}}{|\lambda_n|^{1/2}} \frac{\left\|w_1^n\right\|^{1/2}_{H^3(\Omega_1)}}{|\lambda_n|^{1/4}}\\
&\leq C \frac{\left\|w_3^n\right\|_{H^1(\Omega_1)}}{|\lambda_n|^{1/4}}.
\end{align*}
Therefore,
\begin{equation}\label{ec28}
\frac{1}{|\lambda_n|}\left\langle w_3^n, \partial_{\nu}\Delta w_1^n\right\rangle_{L^2(\Gamma)}\to 0.
\end{equation}
Lemma  \ref{Lemma_01} implies
\begin{align*}
\left|\left\langle \partial_{\nu}w_3^n, \Delta w_1^n\right\rangle_{L^2(\partial\Omega_1)}\right|&\leq\left\|\partial_{\nu}w_3^n\right\|_{L^2(\partial\Omega_1)}\left\|\Delta w_1^n\right\|_{L^2(\partial\Omega_1)}\\
&\leq C \left\|w_3^n\right\|^{1/2}_{H^2(\Omega_1)}\left\|w_3^n\right\|^{1/2}_{H^1(\Omega_1)} \left\|w_1^n\right\|^{1/2}_{H^3(\Omega_1)}\left\|w_1^n\right\|^{1/2}_{H^2(\Omega_1)}.
\end{align*}
By \eqref{tion2} we have
\begin{align*}
\left|\frac{1}{\lambda_n}\left\langle \partial_{\nu}w_3^n,  \Delta w_1^n\right\rangle_{L^2(\partial\Omega_1)}\right|&\leq C \frac{\left\|w_3^n\right\|^{1/2}_{H^2(\Omega_1)}}{|\lambda_n|^{1/2}} \left\|w_3^n\right\|^{1/2}_{H^1(\Omega_1)}\frac{\left\|w_1^n\right\|^{1/2}_{H^3(\Omega_1)}}{|\lambda_n|^{1/4}} \frac{\left\|w_1^n\right\|^{1/2}_{H^2(\Omega_1)}}{|\lambda_n|^{1/4}}\\
&\leq C \left\|w_3^n\right\|^{1/2}_{H^1(\Omega_1)}.
\end{align*}
Hence,
\begin{equation}\label{ec29}
\frac{1}{|\lambda_n|}\left\langle \partial_{\nu}w_3^n, \Delta w_1^n\right\rangle_{L^2(\partial\Omega_1)}\to 0 .
\end{equation}
From \eqref{ec14}, \eqref{ec21} and \eqref{ec26}-\eqref{ec29}, it follows that
\begin{equation}\label{ec30}
\Delta w_1^n\to 0  \ \textup{ in } \ L^2(\Omega_1).
\end{equation}

The limit \eqref{ec11} implies $\left\langle i\lambda_n\rho_1w_2^n+\beta_1\Delta^2w_1^n+\mu\Delta w_3^n, w_2^n\right\rangle_{L^2(\Omega_1)}\to 0 $. Using integration by parts, we have
\begin{align}\label{ec31}
\begin{split}
&\left\langle i\lambda_n\rho_1w_2^n+\beta_1\Delta^2w_1^n+\mu\Delta w_3^n, w_2^n\right\rangle_{L^2(\Omega_1)}=i\lambda_n\rho_1\left\|w_2^n\right\|^2_{L^2(\Omega_1)}\\
&+\beta_1\left\langle \Delta w_1^n, \Delta w_2^n\right\rangle_{L^2(\Omega_1)}+\beta_1\left\langle \partial_{\nu}\Delta w_1^n, w_2^n\right\rangle_{L^2(I)}+\mu\left\langle \Delta w_3, w_2\right\rangle_{L^2(\Omega_1)}.
\end{split}
\end{align}
From \eqref{ec9} and \eqref{ec30}, we get
\begin{equation}\label{ec32}
\frac{1}{|\lambda_n|}\Delta w_2^n\to 0  \ \textup{ in } \ L^2(\Omega_1).
\end{equation}
By Theorem \ref{Th regularity}, the trace theorem and Lemma \ref{Lemma_01}, we get
\begin{align*}
&\left|\left\langle \beta_1\partial_{\nu}\Delta w_1^n, w_2^n\right\rangle_{L^2(I)}\right|\\
&=\left|\left\langle -\beta_2\partial_{\nu}w_4^n-\mu\partial_{\nu}w_3^n, w_2^n\right\rangle_{L^2(I)}\right|\\
&\leq C \left(\left\|\partial_{\nu}w_4^n\right\|_{L^2(I)}+ \left\|\partial_{\nu}w_3^n\right\|_{L^2(I)}\right)\left\|w_2^n\right\|_{L^2(I)}\\
&\leq C \left(\left\|w_4^n\right\|^{1/2}_{H^2(\Omega_2)} \left\|w_4^n\right\|^{1/2}_{H^1(\Omega_2)}+\left\|w_3^n\right\|^{1/2}_{H^2(\Omega_1)} \left\|w_3^n\right\|^{1/2}_{H^1(\Omega_1)}\right)\left\|w_2^n\right\|_{H^1(\Omega_1)}.
\end{align*}
Then,
\begin{align*}\label{ec33}
\begin{split}
&\left|\frac{1}{\lambda_n}\left\langle \beta_1\partial_{\nu}\Delta w_1^n, w_2^n\right\rangle_{L^2(I)}\right|\\
&\leq C \left(\frac{\left\|w_4^n\right\|^{1/2}_{H^2(\Omega_2)}} {|\lambda_n|^{1/2}}\left\|w_4^n\right\|^{1/2}_{H^1(\Omega_2)}+ \frac{\left\|w_3^n\right\|^{1/2}_{H^2(\Omega_1)}}{|\lambda_n|^{1/2}} \left\|w_3^n\right\|^{1/2}_{H^1(\Omega_1)}\right)\frac{\left\|w_2^n\right\|_{H^1(\Omega_1)}}{|\lambda_n|^{1/2}}.
\end{split}
\end{align*}
From \eqref{tion1}, \eqref{ec19} and \eqref{tion2}, it follows that
\begin{equation}\label{ec35}
\frac{1}{|\lambda_n|}\left\langle \partial_{\nu}\Delta w_1^n, w_2^n\right\rangle_{L^2(I)}\to 0 .
\end{equation}
The limits \eqref{ec16} and \eqref{ec32} imply
\begin{equation}\label{ec36}
\frac{1}{|\lambda_n|}\Delta w_3^n\to 0  \ \textup{ in } \ L^2(\Omega_1).
\end{equation}
From \eqref{ec31}, \eqref{ec32}, \eqref{ec35} and \eqref{ec36}, we obtain
\begin{equation}\label{ec37}
w_2^n\to 0 \ \textup{ in } \ L^2(\Omega_1).
\end{equation}
Finally, the limits \eqref{ec14}, \eqref{ec20}, \eqref{ec30} and \eqref{ec37} allow us to write $\left\|w^n\right\|_{\mathcal{H}_0}\to 0 $, which is a contradiction to $\left\|w^n\right\|_{\mathcal{H}_0}=1$ for all $n\in\mathbb{N}$.
\end{proof}


\section{The case of undamped membrane}\label{Lack_exponential_stability}

We now consider the situation when the membrane is undamped, i.e. $m=0$. First, we will prove that the solution of our system is not exponentially stable. The proof follows the ideas from Theorem 3.5 in \cite{munoz2017racke}.

\begin{teor}\label{lack exp stab}
For $\gamma, \rho\geq0$ and $m=0$, the system \eqref{pl-1}-\eqref{inc} is not exponentially stable.
\end{teor}

\begin{proof}
We set $\widetilde{\mathcal H}\coloneqq \{0\}\times\{0\} \times\{0\}\times H^1_0(\Omega_2)\times L^2(\Omega_2)$. Note that $\widetilde{\mathcal H}$ is a Hilbert subspace of $\h$. We define the operator $\widetilde{\mathcal A}$ given by $D(\widetilde{\mathcal A}) = \{0\}\times\{0\} \times\{0\}\times(H^2(\Omega_1)\cap H^1_0(\Omega_1))\times H^1_0(\Omega_1)\subset \widetilde{\mathcal H}$ and
\[ \widetilde{\mathcal A}  w = (0,0,0,w_5, \beta_2/\rho_2 \Delta w_4)^\top.\]
With respect to the fourth and fifth component, $\widetilde{\mathcal H}$ is the first-order system related to the non-damped wave equation for $\widetilde v \coloneqq  \widetilde w_4$
\begin{equation}\label{system}
\begin{cases}
\rho_2\widetilde{v}_{tt}-\beta_2\Delta\widetilde{v}=0 \ \textup{ in } \ \Omega_2\times(0, \infty),\\
\widetilde{v}=0 \ \textup{ on } \ I\times(0, \infty),\\
\widetilde{v}(\cdot, 0)=\widetilde{v}^0, \ \widetilde{v}_t(\cdot, 0)=\widetilde{v}^1 \ \textup{ in } \ \Omega_2.
\end{cases}
\end{equation}
Let $(\widetilde{\mathcal S}(t))_{t\geq0}$ be the $C_0$-semigroup generated by $\widetilde{\mathcal A}$ on $\widetilde{\mathcal H}$. As \eqref{system} contains no damping term, this is a unitary semigroup.
Thus, the essential spectral radius $r_{\text{ess}}(\widetilde{\mathcal S}(t))$ is equal to 1.

We will show that $\mathcal S(t)-\widetilde{\mathcal S}(t) \colon\widetilde{\mathcal H}\longrightarrow\h$ is compact, where $(\mathcal S(t))_{t\ge 0}$ stands for the $C_0$-semigroup generated by $\mathcal A$ (Theorem~\ref{!}). It is enough to prove that $S(t)-\widetilde{S}(t) : \mathcal W\longrightarrow\h$ is compact for some dense subspace $\mathcal W$ of $\widetilde{\mathcal H}$. We define
$$
\mathcal W\coloneqq \{0\}\times\{0\} \times\{0\}\times \mathscr D(\Omega_2) \times \mathscr D(\Omega_2) .
$$
Then $\mathcal W$ is dense in $\widetilde{\mathcal H}$, and obviously $\mathcal W\subset D(\mathcal A)\cap D(\widetilde{\mathcal A})$. For $w_0\in\mathcal W$, we consider
$$
E(t)\coloneqq \frac{1}{2}\|\mathcal S(t)w_0-\widetilde{\mathcal S}(t)w_0\|^2_\h \quad (t\geq0).
$$
Let $w(t) \coloneqq \mathcal S(t)w_0$ and $\widetilde w(t) \coloneqq \widetilde{\mathcal S}(t)w_0$. Then
\begin{equation}\label{5-1}
\begin{aligned}
  E'(t) & = \tfrac{d}{dt} \|w(t)-\widetilde w(t)\|_{\h}^2 = \Re\, \langle w'(t)-\widetilde w'(t), w(t)-\widetilde w(t) \rangle_{\h}  \\
  & = \Re\,\langle \mathcal A w(t) -\widetilde{\mathcal A} \widetilde w(t), w(t)-\widetilde w(t) \rangle_{\h} \\
  & = \Re\,  \langle \mathcal A w(t),w(t)  \rangle_{\h} + \Re\,\langle \widetilde{\mathcal A} \widetilde w(t)\widetilde w(t) \rangle_{\h}  \\
  & \quad   - \Re\,\langle \mathcal A w(t),\widetilde w(t) \rangle_{\h}-\Re\,\langle\widetilde{\mathcal A} \widetilde w(t), w(t) \rangle_{\h}.
\end{aligned}
\end{equation}
From \eqref{2-11} we know
$ \Re\,  \langle \mathcal A w(t),w(t)  \rangle_{\h} \le 0$, and for the undamped wave equation, we obtain $ \Re\,\langle \widetilde{\mathcal A} \widetilde w(t)\widetilde w(t) \rangle_{\h}=0$. Moreover, by the definition of $\mathcal A$ and \eqref{2-11a}, we see that
\begin{align*}
 \langle \mathcal Aw(t),\widetilde w(t)\rangle_{\h}   & = \langle\mathbb Aw(t),\widetilde w(t)\rangle_{\mathcal H_\gamma'\times \h}\\
  & = -\beta_2\langle\nabla w_4(t),\nabla\widetilde w_5(t)\rangle_{L^2(\Omega_2)} + \beta_2 \langle \nabla w_5(t),\nabla\widetilde w_4(t)\rangle_{L^2(\Omega_2)}.
\end{align*}
With integration by parts, we obtain
\begin{align*}
 \langle\widetilde{\mathcal A} \widetilde w(t),& w(t) \rangle_{\h} = \beta_2\langle\nabla\widetilde w_5(t), \nabla w_4(t)\rangle_{L^2(\Omega_2)} + \rho_2 \langle \tfrac{\beta_2}{\rho_2} \Delta\widetilde w_4(t), w_5(t)\rangle_{L^2(\Omega_2)}\\
  & = \beta_2 \langle\nabla\widetilde w_5(t), \nabla w_4(t)\rangle_{L^2(\Omega_2)} - \beta_2 \langle  \nabla\widetilde w_4(t), \nabla w_5(t)\rangle_{L^2(\Omega_2)}\\
  & \quad  + \langle\partial_\nu\widetilde w_4(t),w_5(t)\rangle_{L^2(I)}.
\end{align*}
Taking the real part in the last two equalities and inserting this into \eqref{5-1}, we see that
\[ E'(t) \le \langle\partial_\nu\widetilde w_4(t),w_5(t)\rangle_{L^2(I)} =  \langle\partial_\nu\widetilde w_4(t),w_2(t)\rangle_{L^2(I)},\]
where we used in the last equality that $w_5(t)=w_2(t)$ on $I$ because $w(t)\in D(\mathcal A)$. Therefore, noting $E(0)=0$, we have
\begin{equation}\label{integral}
E(t)\le \beta_2\Re \int_0^t\langle\partial_{\nu}\widetilde{w_4}(s), w_2(s)\rangle_{L^2(I)}ds.
\end{equation}
Let $(w_0^{k})_{k\in\N}\subset\mathcal W$ be a bounded sequence, and let $w^{k}(t) \coloneqq \mathcal S(t)w_0^{k}$ and $\widetilde w ^{k}(t) \coloneqq \widetilde {\mathcal S}(t)w_0^{k}$. As $\widetilde w^{k}\in C([0,\infty),D(\widetilde{\mathcal A}))$, the sequence $(\partial_{\nu}\widetilde w_4^{k})_{k\in\N}\subset L^2([0, t], L^2(I))$ is uniformly bounded.
Therefore there exists a subsequence which will again be denoted by $(\widetilde w_4^k)_{k\in\N}$ such that
$(\partial_{\nu}\widetilde w_4^k)_{k\in\N}$ converges weakly in $L^2([0, t], L^2(I))$. Similarly, due to $w^k\in C^1([0, \infty), \h)$, we have that the sequences $(w_2^k)_{k\in\N}\subset L^2([0, t], H^2(\Omega_1))$ and
$(\partial_t w_2^k)_{k\in\N}\subset L^2([0, t], L^2(\Omega_1))$ are both uniformly bounded. By the lemma of Aubin-Lions, there exists a
subsequence  which will again be denoted by $(w^k)_{k\in\N}$, such that $(w_2^k)_{k\in\N}\subset L^2([0, t], H^1(\Omega_1))$
converges. As the trace to the boundary is continuous from $H^1(\Omega_1)$ to $L^2(I)$, also the sequence
$(w_2^k)_{k\in\N}\subset L^2([0, t], L^2(I))$ is convergent. For $k, \ell\in\N$ we now denote by
$$
E^{k\ell}(t)\coloneqq \frac{1}{2}\|\mathcal S(t)(w_0^k-w_0^\ell)-\widetilde{\mathcal S}(t)(w_0^k-w_0^\ell)\|^2_\h \quad (t\geq0).
$$
By \eqref{integral} we have that
$$
E^{k\ell}(t)\leq \beta_2\big| \langle\partial_{\nu}\widetilde{w}_4^{k\ell}, w_2^{k\ell}\rangle_{L^2([0, t], L^2(I))}\big|\to 0\;(k, \ell\to\infty),
$$
where ${w}^{k\ell}\coloneqq {\mathcal S}(t)(w_0^k - w_0^\ell)$ and $\widetilde{w}^{k\ell}\coloneqq \widetilde{\mathcal S}(t)(w_0^k - w_0^\ell)$ for $k,\ell\in\N$. Therefore, $((\mathcal S(t)-\widetilde{\mathcal S}(t))w_0^k)_{k\in\N}$ is a Cauchy sequence in $\h$ and thus
convergent. This shows the compactness of $\mathcal S(t)-\widetilde{\mathcal S}(t) \colon \mathcal W\to\h$. Therefore, $\mathcal S(t)-\widetilde{\mathcal S}(t) \colon \widetilde{\mathcal H}\to\h$ is compact. As $r_{\operatorname{ess}}(\widetilde{\mathcal S}(t))=1$, Theorem 3.3
in \cite{munoz2017racke} implies that $r_{\operatorname{ess}}(\mathcal S(t))=1$, and thus $(\mathcal S(t))_{t\geq0}$ is not exponentially stable.
\end{proof}

Although the last result tells us that there  is no exponential stability in the case of an undamped membrane, we will now show that the system decays polynomially under the following geometric condition: There exists some $x_0\in\R^2$ such that
\begin{equation}\label{geometric condition}
q(x)\cdot\nu(x)\leq0  \; (x\in I), \quad \text{ where } q(x) \coloneqq  x-x_0\;(x\in \overline{\Omega_2}).
\end{equation}

\begin{teor}\label{teorcondition}
Let $m=0$, $\rho>0$, $\gamma\ge 0$ and assume that the geometrical condition \eqref{geometric condition} is satisfied. Then, the semigroup $(\mathcal{S}(t))_{t\geq0}$ generated by $\mathcal A$ decays polynomially,  i.e., there exist constants $\alpha,C>0$ such that
$$
\left\|\mathcal{S}(t)w_0\right\|_{\h}\leq Ct^{-\alpha}\left\|w_0\right\|_{D (\mathcal A)}
$$
for all $t>0$ and $w_0\in D (\mathcal A)$.
\end{teor}

\begin{proof}
By Lemma 5.2 in \cite{munoz2017racke}, the semigroup is polynomially stable if $i\R\subset\rho(\mathcal A)$ and if there exist
 $C>0$, $\lambda_0>0$, and $\beta>0, \,\beta'\ge0$ with
\begin{equation}\label{5-2}
\|(i\lambda  - \mathcal A)^{-1}f\|_{\h}\leq C|\lambda|^\beta\|\mathcal A^{\beta'}f\|_{\h}\quad (f\in D(\mathcal A^{\beta'}),\,\lambda\in\R,\, |\lambda|>\lambda_0).
\end{equation}
First, let $\gamma>0$. We will show \eqref{5-2} with $\beta'=1$.
Let $\lambda_0>0$ and $\lambda\in\R$ with $|\lambda|>\lambda_0$. Let $w\in D(\mathcal{A}^2)$ and  $f:=(i\lambda  -\mathcal A)w$. Then $f\in D({\mathcal A})$, and
\begin{align}
i\lambda w_1- w_2&=f_1,\label{2equa}\\
i\lambda\rho_1 w_2-i\lambda\gamma\Delta w_2+\beta_1\Delta^2 w_1+\mu\Delta w_3-\rho\Delta w_2&=\rho_1f_2-\gamma\Delta f_2,\label{4equa}\\
i\lambda\rho_0 w_3-\mu\Delta w_2-\beta_0\Delta w_3&=\rho_0f_3,\label{6equa}\\
i\lambda w_4- w_5&=f_4,\label{3equa}\\
i\lambda\rho_2 w_5-\beta_2\Delta w_4&=\rho_2f_5.\label{5equa}
\end{align}
Replacing \eqref{2equa} into \eqref{4equa} and \eqref{6equa}, we have
\begin{equation}\label{7equa}
\begin{aligned}
-\lambda^2\rho_1 w_1 +\lambda^2\gamma\Delta w_1&+\beta_1\Delta^2 w_1+\mu\Delta w_3-i\lambda\rho\Delta w_1\\
&=i\lambda\rho_1f_1-i\lambda\gamma\Delta f_1-\rho\Delta f_1+\rho_1f_2-\gamma\Delta f_2,
\end{aligned}
\end{equation}
and
\begin{equation}\label{8equa}
i\lambda\rho_0 w_3-i\lambda\mu\Delta w_1-\beta_0\Delta w_3=-\mu\Delta f_1+\rho_0f_3,
\end{equation}
respectively. Replacing \eqref{3equa} into \eqref{5equa}, we get
\begin{equation}\label{9equa}
-\lambda^2\rho_2 w_4-\beta_2\Delta w_4=i\lambda\rho_2f_4+\rho_2f_5.
\end{equation}
Multiplying \eqref{7equa} by $-\overline{ w_1}$, \eqref{8equa} by $i\frac{1}{\lambda}\overline{ w_3}$ and \eqref{9equa} by $-\overline{ w_4}$, integrating and  adding the resulting equalities, we obtain
\begin{align*}
&\lambda^2\left(\rho_1\left\| w_1\right\|^2_{L^2(\Omega_1)}+\rho_2\left\| w_4\right\|^2_{L^2(\Omega_2)}\right)-\rho_0\left\| w_3\right\|^2_{L^2(\Omega_1)}-\beta_1\left\langle \Delta^2 w_1,  w_1\right\rangle_{L^2(\Omega_1)}\\
&+i\lambda\rho\left\langle \Delta w_1,  w_1\right\rangle_{L^2(\Omega_1)}-\lambda^2\gamma\left\langle \Delta w_1, w_1\right\rangle_{L^2(\Omega_1)}-\mu\left\langle \Delta w_3,  w_1\right\rangle_{L^2(\Omega_1)}\\
&+\beta_2\left\langle \Delta w_4,  w_4\right\rangle_{L^2(\Omega_2)}+\mu\left\langle \Delta w_1, w_3\right\rangle_{L^2(\Omega_1)}-i\beta_0\lambda^{-1}\left\langle \Delta w_3,  w_3\right\rangle_{L^2(\Omega_1)}\\
&=-i\lambda\rho_1\left\langle f_1,  w_1\right\rangle_{L^2(\Omega_1)}+i\lambda\gamma\left\langle \Delta f_1, w_1\right\rangle_{L^2(\Omega_1)}+\rho\left\langle \Delta f_1,  w_1\right\rangle_{L^2(\Omega_1)}-\rho_1\left\langle f_2, w_1\right\rangle_{L^2(\Omega_1)}\\
&+\gamma\left\langle \Delta f_2,  w_1\right\rangle_{L^2(\Omega_1)}-\left\langle i\lambda\rho_2f_4+\rho_2f_5, w_4\right\rangle_{L^2(\Omega_2)}-\left\langle i\mu\lambda^{-1}\Delta f_1-i\rho_0\lambda^{-1}f_3,  w_3\right\rangle_{L^2(\Omega_1)}.
\end{align*}
Integrating by parts and using the transmission conditions, we obtain
\begin{align*}
&\lambda^2\left(\rho_1\left\| w_1\right\|^2_{L^2(\Omega_1)}+\rho_2\left\| w_4\right\|^2_{L^2(\Omega_2)}+\gamma\left\|\nabla w_1\right\|^2_{L^2(\Omega_1)}\right)-\beta_1\left\|\Delta w_1\right\|^2_{L^2(\Omega_1)}\\
&-\rho_0\left\| w_3\right\|^2_{L^2(\Omega_1)}-i\lambda\rho\left\|\nabla w_1\right\|^2_{L^2(\Omega_1)}+i2\mu\textup{Im}\left\langle \Delta w_1,  w_3\right\rangle_{L^2(\Omega_1)}-\beta_2\left\|\nabla w_4\right\|^2_{L^2(\Omega_2)}\\
&+i\frac{\beta_0}{\lambda}\left\|\nabla w_3\right\|^2_{L^2(\Omega_1)}+i\frac{\kappa\beta_0}{\lambda}\left\| w_3\right\|^2_{L^2(\Gamma)}=-(\rho+i\lambda\gamma)\left\langle \nabla f_1, \nabla w_1\right\rangle_{L^2(\Omega_1)}\\
&-\gamma\left\langle \nabla f_2, \nabla w_1\right\rangle_{L^2(\Omega_1)}-\left\langle i\lambda\rho_1f_1+\rho_1f_2, w_1\right\rangle_{L^2(\Omega_1)}-\left\langle \rho_2f_5+i\lambda\rho_2f_4,  w_4\right\rangle_{L^2(\Omega_2)}\\
&+i\frac{\mu}{\lambda}\left\langle \nabla f_1, \nabla w_3\right\rangle_{L^2(\Omega_1)}+i\frac{\rho_0}{\lambda}\left\langle f_3, w_3\right\rangle_{L^2(\Omega_1)}.
\end{align*}
Taking real part in the previous equation, we get
\begin{align*}
&\beta_1\left\|\Delta w_1\right\|^2_{L^2(\Omega_1)}+\beta_2\left\|\nabla w_4\right\|^2_{L^2(\Omega_2)}+\rho_0\left\| w_3\right\|^2_{L^2(\Omega_1)}\\
&\leq\lambda^2\left(\rho_1\left\| w_1\right\|^2_{L^2(\Omega_1)}+\rho_2\left\| w_4\right\|^2_{L^2(\Omega_2)}+\gamma\left\|\nabla w_1\right\|^2_{L^2(\Omega_1)}\right)\\
&\quad+\left(\rho\left\|\nabla f_1\right\|_{L^2(\Omega_1)}+|\lambda|\gamma\left\|\nabla f_1\right\|_{L^2(\Omega_1)}+\gamma\left\|\nabla f_2\right\|_{L^2(\Omega_1)}\right)\left\|\nabla w_1\right\|_{L^2(\Omega_1)}\\
&\quad+\rho_1\left(|\lambda|\left\|f_1\right\|_{L^2(\Omega_1)}+\left\|f_2\right\|_{L^2(\Omega_1)}\right)\left\| w_1\right\|_{L^2(\Omega_1)}+\frac{\mu}{|\lambda|}\left\|\nabla f_1\right\|_{L^2(\Omega_1)}\left\|\nabla w_3\right\|_{L^2(\Omega_1)}\\
&\quad+\rho_2\left(\left\|f_5\right\|_{L^2(\Omega_2)}+|\lambda|\left\|f_4\right\|_{L^2(\Omega_2)}\right)\left\| w_4\right\|_{L^2(\Omega_2)}+\frac{\rho_0}{|\lambda|}\left\|f_3\right\|_{L^2(\Omega_1)}\left\| w_3\right\|_{L^2(\Omega_1)}.
\end{align*}
By Lemma~\ref{2.1}, \eqref{dissip} and $|\lambda|\ge \lambda_0$, we obtain
\begin{align}\label{10equa}
\begin{split}
\beta_1\|\Delta &w_1\|^2_{L^2(\Omega_1)}+\beta_2\|\nabla w_4\|^2_{L^2(\Omega_2)}+\rho_0\| w_3\|^2_{L^2(\Omega_1)}\le  C \Big(\lambda^2\| w_1\|^2_{H^1(\Omega_1)}\\
& \quad +\lambda^2\| w_4\|^2_{L^2(\Omega_2)}+|\lambda|\left\|\omega\right\|_{\h}\left\|f\right\|_{\h} +|\lambda|^{-1}\left\|\omega\right\|^{1/2}_{\h}\left\|f\right\|^{3/2}_{\h}\Big).
\end{split}
\end{align}
Since
\begin{align*}
 \| w_2 \|^2_{L^2(\Omega_1)} & \leq4 (\lambda^2 \| w_1 \|^2_{L^2(\Omega_1)}+ \|f_1 \|^2_{L^2(\Omega_1)} )\leq C (\lambda^2 \| w_1 \|^2_{H^1(\Omega_1)}+ \|f \|^2_{\h} ),\\
 \|\nabla w_2 \|^2_{L^2(\Omega_1)}& \leq\tfrac{1}{\rho} \|\omega \|_{\h} \|f\|_{\h},\\
 \| w_5 \|^2_{L^2(\Omega_2)}& \leq4 (\lambda^2 \| w_4 \|^2_{L^2(\Omega_2)}+ \|f_4 \|^2_{L^2(\Omega_2)} )\leq C (\lambda^2 \| w_4 \|^2_{L^2(\Omega_2)}+ \|f \|^2_{\h} ),
\end{align*}
we get
\begin{align}\label{11equa}
\begin{split}
&\rho_1\left\| w_2\right\|^2_{L^2(\Omega_1)}+\gamma\left\|\nabla w_2\right\|^2_{L^2(\Omega_1)}+\rho_2\left\| w_5\right\|^2_{L^2(\Omega_2)}\\
& \quad \leq C \left(\lambda^2\left\| w_1\right\|^2_{H^1(\Omega_1)}+\lambda^2\left\| w_4\right\|^2_{L^2(\Omega_2)}+\left\|\omega\right\|_{\h}\left\|f\right\|_{\h}+\left\|f\right\|^2_{\h}\right).
\end{split}
\end{align}
Poincar\'e's inequality implies
\begin{align}\label{12equa}
\lambda^2\left\| w_1\right\|^2_{H^1(\Omega_1)}&=\left\| w_2+f_1\right\|^2_{H^1(\Omega_1)}\leq C \left(\left\|\nabla w_2\right\|^2_{L^2(\Omega_1)}+\left\|f_1\right\|^2_{H^1(\Omega_1)}\right)\nonumber\\
&\leq C \left(\left\|\omega\right\|_{\h}\left\|f\right\|_{\h}+\left\|f\right\|^2_{\h}\right).
\end{align}
From \eqref{10equa}-\eqref{12equa}, it follows that
\begin{align}\label{13equa}
\begin{split}
\left\|\omega\right\|^2_{\h}\leq C \Big(\lambda^2\left\| w_4\right\|^2_{L^2(\Omega_2)}&+|\lambda|\left\|\omega\right\|_{\h}\left\|f\right\|_{\h}\\
&+|\lambda|^{-1}\left\|\omega\right\|^{1/2}_{\h}\left\|f\right\|^{3/2}_{\h}+\left\|f\right\|^2_{\h}\Big).
\end{split}
\end{align}

Now we will prove that
\begin{equation}\label{14equa}
\lambda^2\left\| w_4\right\|^2_{L^2(\Omega_2)}\leq C \Big(|\lambda|\left\|\omega\right\|_{\h}\left\|f\right\|_{\h}+\left\|f\right\|^2_{\h}+\beta_2\int_I|\partial_{\nu} w_4(q\nabla\overline{ w_4})|dS\Big).
\end{equation}
In fact, using Rellich's identity (see equation (2.5) from \cite{Mitidieri1993}), we have the following equality
\begin{equation}\label{15equa}
\textup{Re}\int_{\Omega_2}\Delta w_4(q\nabla\overline{ w_4})dx=-\textup{Re}\int_I\left[\partial_{\nu} w_4(q\nabla\overline{ w_4})-\frac{1}{2}(q\cdot\nu)|\nabla w_4|^2\right]dS.
\end{equation}
Multiplying \eqref{9equa} by $q\nabla\overline{ w_4}$ and integrating, we get
$$
-\lambda^2\rho_2\int_{\Omega_2} w_4(q\nabla\overline{ w_4})dx-\beta_2\int_{\Omega_2}\Delta w_4(q\nabla\overline{ w_4})dx=\int_{\Omega_2}(i\lambda\rho_2f_4+\rho_2f_5)(q\nabla\overline{ w_4})dx.
$$
Taking real part and using \eqref{15equa}, we see that
\begin{align*}
&-\lambda^2\rho_2\textup{Re}\int_{\Omega_2} w_4(q\nabla\overline{ w_4})dx+\beta_2\textup{Re}\int_I\left[\partial_{\nu} w_4(q\nabla\overline{ w_4})-\frac{1}{2}(q\cdot\nu)|\nabla w_4|^2\right]dS\\
&=\textup{Re}\int_{\Omega_2}(i\lambda\rho_2f_4+\rho_2f_5)(q\nabla\overline{ w_4})dx.
\end{align*}
Using integration by parts and the identity $q\nabla w_4=\textup{div}(q w_4)-2 w_4$, it holds
\begin{align*}
\int_{\Omega_2} w_4(q\nabla\overline{ w_4})dx&=\int_{\Omega_2} w_4\left(\textup{div}(q\overline{ w_4})-2\overline{ w_4}\right)dx\\
&=\int_{\Omega_2} w_4\textup{div}(q\overline{ w_4})dx-2\int_{\Omega_2} w_4\overline{ w_4}dx\\
&=-\int_{\Omega_2}\nabla w_4\cdot q\overline{ w_4}dx-\int_I w_4q\overline{ w_4}\cdot\nu dS-2\left\| w_4\right\|^2_{L^2(\Omega_2)}
\end{align*}
and therefore
$$
\int_{\Omega_2} w_4(q\nabla\overline{ w_4})dx+\overline{\int_{\Omega_2} w_4(q\nabla\overline{ w_4})dx}=-\int_I(q\nu)| w_4|^2dS-2\left\| w_4\right\|^2_{L^2(\Omega_2)}.
$$
Thus,
$$
\textup{Re}\int_{\Omega_2} w_4(q\nabla\overline{ w_4})dx=-\left\| w_4\right\|^2_{L^2(\Omega_2)}-\frac{1}{2}\int_I(q\nu)| w_4|^2dS.
$$
In consequence,
\begin{align*}
\lambda^2\rho_2\left\| w_4\right\|^2_{L^2(\Omega_2)}&=-\lambda^2\rho_2\textup{Re}\int_{\Omega_2} w_4(q\nabla\overline{ w_4})dx-\frac{1}{2}\lambda^2\rho_2\int_I(q\nu)| w_4|^2dS\\
&=\textup{Re}\int_{\Omega_2}(i\lambda\rho_2f_4+\rho_2f_5)(q\nabla\overline{ w_4})dx-\beta_2\textup{Re}\int_I\partial_{\nu} w_4(q\nabla\overline{ w_4})dS\\
&\quad+\frac{1}{2}\beta_2\int_I(q\nu)|\nabla w_4|^2dS-\frac{1}{2}\lambda^2\rho_2\int_I(q\nu)| w_4|^2dS.
\end{align*}
Due to $q\cdot\nu\leq0$ on $I$, $ w_1= w_4$ on $I$, trace theorem and \eqref{12equa}, we obtain that
\begin{align*}
\lambda^2\rho_2\left\| w_4\right\|^2_{L^2(\Omega_2)}&\leq C \Big(|\lambda|\left\|\omega\right\|_{\h}\left\|f\right\|_{\h}+\beta_2\left|\int_I\partial_{\nu} w_4(q\nabla\overline{ w_4})dS\right|+\lambda^2\left\| w_1\right\|^2_{L^2(I)}\Big)\\
&\leq  \Big(|\lambda|\left\|\omega\right\|_{\h}\left\|f\right\|_{\h}+\beta_2\int_I|\partial_{\nu} w_4(q\nabla\overline{ w_4})|dS+\lambda^2\left\| w_1\right\|^2_{H^1(\Omega_1)}\Big)\\
&\leq C\Big(|\lambda|\left\|\omega\right\|_{\h}\left\|f\right\|_{\h}+\left\|f\right\|^2_{\h}+\beta_2\int_I|\partial_{\nu} w_4(q\nabla\overline{ w_4})|dS\Big).
\end{align*}

Next, we will show that for any $\varepsilon>0$ there exists a constant $C_{\varepsilon}>0$ such that
\begin{equation}\label{16equa}
\beta_2\int_I|\partial_{\nu} w_4(q\nabla\overline{ w_4})|dS\leq\varepsilon\left\|\omega\right\|^2_{\h}+C_{\varepsilon}|\lambda|^{48}\left\|{\mathcal A} f\right\|^2_{\h}.
\end{equation}
Indeed, the  transmission conditions imply
\begin{align}\label{17equa}
\beta_2\int_I|\partial_{\nu} w_4(q\nabla\overline{ w_4})|dS&\leq\left\|\beta_1\partial_{\nu}(\Delta w_1)+\mu\partial_{\nu} w_3\right\|_{L^2(I)}\left\|q\nabla\overline{ w_4}\right\|_{L^2(I)}\nonumber\\
&\leq C \left(\left\|\partial_{\nu}(\Delta w_1)\right\|_{L^2(I)}+\left\|\partial_{\nu} w_3\right\|_{L^2(I)}\right)\left\|\nabla w_4\right\|_{L^2(I)}.
\end{align}
From \eqref{5equa} it holds that $\Delta w_4=\frac{1}{\beta_2}(i\lambda\rho_2 w_5-\rho_2f_5)$. Because of $ w_1= w_4$ on $I$,   elliptic regularity for the Dirichlet-Laplace operator (Theorem~\ref{Cor Regu boundaryProbNOnHom})
yields
\begin{align}\label{18equa}
\left\| w_4\right\|_{H^2(\Omega_2)}&\leq C\left(\left\|i\lambda\rho_2 w_5-\rho_2f_5\right\|_{L^2(\Omega_2)}+\left\| w_1\right\|_{H^{3/2}(I)}\right)\nonumber\\
&\leq C \left(|\lambda|\left\| w_5\right\|_{L^2(\Omega_2)}+\left\|f_5\right\|_{L^2(\Omega_2)}+\left\| w_1\right\|_{H^2(\Omega_1)}\right)\nonumber\\
&\leq C \left(|\lambda|\left\|\omega\right\|_{\h}+\left\|f\right\|_{\h}\right).
\end{align}
\textcolor{black}{Applying Lemma \ref{Lemma_01} to $\nabla w_4$}, we see that
\begin{align}\label{19equa}
\textcolor{black}{\left\| \nabla w_4\right\|_{L^2(I)}}&\leq C \left\| w_4\right\|^{1/2}_{H^2(\Omega_2)}\left\| w_4\right\|^{1/2}_{H^1(\Omega_2)}\nonumber\\
&\leq C \left(|\lambda|\left\|w\right\|_{\h}+\left\|f\right\|_{\h}\right)^{1/2}\left\|w\right\|^{1/2}_{\h}\nonumber\\
&\leq C \left(|\lambda|^{1/2}\left\|w\right\|_{\h}+\left\|w\right\|^{1/2}_{\h}\left\|f\right\|^{1/2}_{\h}\right).
\end{align}
Note that $ w_3$ belongs to $H^2(\Omega_1)$ and  is a solution of the  problem
$$
\begin{cases}
\Delta w_3=\dfrac{\rho_0}{\beta_0}(i\lambda w_3-\rho_0^{-1}\mu\Delta w_2-f_5)\eqqcolon h^*\in L^2(\Omega_1),\\
\partial_{\nu} w_3+\kappa  w_3=0 \ \textup{ on } \ \Gamma,\\
 w_3=0 \ \textup{ on } \ I.
\end{cases}
$$
By Lemma \ref{Lem ParameterEllipOps} and Theorem \ref{Cor Regu boundaryProbNOnHom}, there exist $\eta_0>0$ such that
\begin{align*}
\left\| w_3\right\|_{H^2(\Omega_1)}&\leq C \left\|\eta_0 w_3-h^*\right\|_{L^2(\Omega_1)}\\
&\leq C \biggl(\left\| w_3\right\|_{L^2(\Omega_1)}+|\lambda|\left\| w_3\right\|_{L^2(\Omega_1)}+\left\|\Delta w_2\right\|_{L^2(\Omega_1)}+\left\|f_5\right\|_{L^2(\Omega_1)}\biggr).
\end{align*}
From \eqref{equality1} and the last inequality, it follows that
\begin{align*}
\left\| w_3\right\|_{H^2(\Omega_1)}&\leq C \biggl(\left\| w_3\right\|_{H^1_D}+|\lambda|\left\| w_3\right\|_{H^1_D}+|\lambda|\left\|\Delta w_1\right\|_{L^2(\Omega_1)}\\
&\phantom{aaaaaa}+\left\|\Delta f_1\right\|_{L^2(\Omega_1)}+\left\|f_5\right\|_{L^2(\Omega_1)}\biggr).
\end{align*}
This and \eqref{dissip} imply
\begin{align}\label{ecua10}
\left\| w_3\right\|_{H^2(\Omega_1)}\leq C \biggr[(1+|\lambda|)\left\|\omega\right\|^{1/2}_{\h}\left\|F\right\|^{1/2}_{\h}
+|\lambda|\left\|\omega\right\|_{\h}+\left\|F\right\|_{\h}\biggr].
\end{align}
Due to \eqref{7equa}, $ w_1$ satisfies the equation
$$
(\eta_0+\Delta^2) w_1=\eta_0 w_1+\beta_1^{-1}z
$$
with $z\coloneqq (\lambda^2\rho_1-\lambda^2\gamma\Delta+i\lambda\rho\Delta) w_1-\mu\Delta w_3+(i\lambda\rho_1-i\lambda\gamma\Delta-\rho\Delta )f_1+(\rho_1-\gamma\Delta)f_2$ and $\eta_0>0$. Note that, if $g\coloneqq {\mathcal A} f$ then $g_1=f_2$ and therefore
\begin{equation}\label{20equa}
\left\|\Delta f_2\right\|_{L^2(\Omega_1)}=\left\|\Delta g_1\right\|_{L^2(\Omega_1)}\leq C \left\|g\right\|_{\h}=C\left\|{\mathcal A} f\right\|_{\h}.
\end{equation}
From Theorem~\ref{Cor Regu boundaryProbNOnHom}, the transmission conditions,  inequalities \eqref{ecua10}, \eqref{18equa}, \eqref{20equa} and $0\in\rho({\mathcal A})$, we obtain
\begin{align*}
\left\| w_1\right\|_{H^4(\Omega_1)}&\leq C \Big(\left\|\eta_0 w_1+\beta_1^{-1}z\right\|_{L^2(\Omega_1)}+\left\|\beta_1^{-1}(-\beta_2\partial_{\nu} w_4-\mu\partial_{\nu} w_3)\right\|_{H^{1/2}(I)}\Big)\\
&\leq C \Big(\lambda^2\left\| w\right\|_{\h}+|\lambda|\left\|w\right\|^{1/2}_{\h}\left\|f\right\|^{1/2}_{\h} +|\lambda|\left\|f\right\|_{\h}+|\lambda|\left\|{\mathcal A} f\right\|_{\h}\Big)\\
&\leq C |\lambda|\Big(|\lambda|\left\| w\right\|_{\h}+\left\|w\right\|^{1/2}_{\h}\left\|{\mathcal A} f\right\|^{1/2}_{\h}+\left\|{\mathcal A} f\right\|_{\h}\Big).
\end{align*}
We have from \eqref{12equa} that
$$
\left\| w_1\right\|_{H^1(\Omega_1)}\leq C |\lambda|^{-1}\Big(\left\|w\right\|^{1/2}_{\h}\left\|f\right\|^{1/2}_{\h}+\left\|f\right\|_{\h}\Big).
$$
\textcolor{black}{Applying Lemma \ref{Lemma_01} to $\partial_\nu(\Delta w_1)$}, using interpolation inequality and $0\in\rho({\mathcal A})$, we obtain that
\begin{equation}\label{21equa}
\begin{aligned}
\textcolor{black}{\| \partial_\nu (\Delta w_1)}&\textcolor{black}{\|_{L^2(I)}}\leq C \left\| w_1\right\|^{5/6}_{H^4(\Omega_1)}\left\| w_1\right\|^{1/6}_{H^1(\Omega_1)}\\
&\leq C |\lambda|^{5/6}\big(|\lambda|\left\| w\right\|_{\h}+\left\| w\right\|^{1/2}_{\h}\left\|{\mathcal A} f\right\|^{1/2}_{\h}+\left\|{\mathcal A} f\right\|_{\h}\big)^{5/6}\\
&\quad\cdot|\lambda|^{-1/6}\big(\left\| w\right\|^{1/2}_{\h}\left\|{\mathcal A} f\right\|^{1/2}_{\h}+\left\|{\mathcal A} f\right\|_{\h}\big)^{1/6}\\
&\leq C |\lambda|^{2/3}\big(|\lambda|^{5/6}\left\| w\right\|^{5/6}_{\h}\left\|{\mathcal A} f\right\|^{1/6}_{\h}+|\lambda|^{5/6}\left\| w\right\|^{11/12}_{\h}\left\|{\mathcal A} f\right\|^{1/12}_{\h}+\left\|{\mathcal A} f\right\|_{\h}\\
&\quad+\left\| w\right\|^{5/12}_{\h}\left\|{\mathcal A} f\right\|^{7/12}_{\h}+\left\| w\right\|^{1/2}_{\h}\left\|{\mathcal A} f\right\|^{1/2}_{\h}+\left\| w\right\|^{1/12}_{\h}\left\|{\mathcal A} f\right\|^{11/12}_{\h}\big).
\end{aligned}
\end{equation}
From \eqref{19equa} and \eqref{21equa}, it follows that
\begin{align}\label{22equa}
\begin{split}
 \textcolor{black}{\| \partial_\nu(\Delta w_1)}&\textcolor{black}{ \|_{L^2(I)}\left\| \nabla w_4\right\|_{L^2(I)}}\\
&\leq C\big(|\lambda|^2\left\| w\right\|^{11/6}_{\h}\left\|{\mathcal A} f\right\|^{1/6}_{\h}+|\lambda|^{3/2}\left\| w\right\|^{4/3}_{\h}\left\|{\mathcal A} f\right\|^{2/3}_{\h}\\
&\quad\quad\;+|\lambda|^2\left\| w\right\|^{23/12}_{\h}\left\|{\mathcal A} f\right\|^{1/12}_{\h}+|\lambda|^{3/2}\left\| w\right\|^{17/12}_{\h}\left\|{\mathcal A} f\right\|^{7/12}_{\h}\\
&\quad\quad\;+|\lambda|^{7/6}\left\| w\right\|^{17/12}_{\h}\left\|{\mathcal A} f\right\|^{7/12}_{\h}+|\lambda|^{2/3}\left\| w\right\|^{11/12}_{\h}\left\|{\mathcal A} f\right\|^{13/12}_{\h}\\
&\quad\quad\;+|\lambda|^{7/6}\left\| w\right\|^{3/2}_{\h}\left\|{\mathcal A} f\right\|^{1/2}_{\h}+|\lambda|^{2/3}\left\| w\right\|_{\h}\left\|{\mathcal A} f\right\|_{\h}\\
&\quad\quad\;+|\lambda|^{7/6}\left\| w\right\|_{\h}\left\|{\mathcal A} f\right\|_{\h}+|\lambda|^{2/3}\left\| w\right\|^{1/2}_{\h}\left\|{\mathcal A} f\right\|^{3/2}_{\h}\\
& \quad\quad\;+|\lambda|^{7/6}\left\| w\right\|^{13/12}_{\h}\left\|{\mathcal A} f\right\|^{11/12}_{\h}+|\lambda|^{2/3}\left\| w\right\|^{7/12}_{\h}\left\|{\mathcal A} f\right\|^{17/12}_{\h}\big).
\end{split}
\end{align}
\textcolor{black}{Applying Lemma \ref{Lemma_01} to $\partial_\nu w_3$ and using \eqref{dissip} and \eqref{ecua10}, we get}
\begin{align}\label{23a_equa}
 \textcolor{black}{\Vert }&\textcolor{black}{\partial_\nu w_3 \Vert_{L^2(I)}}\nonumber\\
 & \textcolor{black}{\leq C\Vert w_3\Vert_{H^2(\Omega_1)}^{1/2}\Vert w_3\Vert_{H^1(\Omega_1)}^{1/2}}\nonumber\\
 & \textcolor{black}{\leq C\Big( \Vert w\Vert_{\h}^{1/2}\Vert f\Vert_{\h}^{1/2} + \vert \lambda\vert^{1/2}\Vert w\Vert_{\h}^{1/2}\Vert f\Vert_{\h}^{1/2} + \vert \lambda\vert^{1/2}\Vert w\Vert_{\h}^{3/4}\Vert f\Vert_{\h}^{1/4} + \Vert w\Vert_{\h}^{1/4}\Vert f\Vert_{\h}^{3/4}\Big).}
\end{align}
\textcolor{black}{Then, from  \eqref{19equa} and \eqref{23a_equa} we obtain}
\begin{align}\label{23equa}
\begin{split}
&\textcolor{black}{\left\| \partial_\nu w_3\right\|_{L^2(I)}\left\| \nabla w_4\right\|_{L^2(I)}}\\
&\quad\leq C\big(|\lambda|\left\| w\right\|^{3/2}_{\h}\left\|{\mathcal A} f\right\|^{1/2}_{\h}+|\lambda|^{1/2}\left\| w\right\|_{\h}\left\|{\mathcal A} f\right\|_{\h}\\
&\qquad\quad\;+|\lambda|\left\| w\right\|^{7/4}_{\h}\left\|{\mathcal A} f\right\|^{1/4}_{\h}+|\lambda|^{1/2}\left\| w\right\|^{5/4}_{\h}\left\|{\mathcal A} f\right\|^{3/4}_{\h}\\
&\qquad\qquad\;\;\;+|\lambda|^{1/2}\left\| w\right\|^{5/4}_{\h}\left\|{\mathcal A} f\right\|^{3/4}_{\h}+\left\| w\right\|^{3/4}_{\h}\left\|{\mathcal A} f\right\|^{5/4}_{\h}\big).
\end{split}
\end{align}
\textcolor{black}{Considering \eqref{22equa}, \eqref{23equa} and Young's inequality, we observe that the worst term in the estimate of the right side of \eqref{17equa} is
$$ \vert\lambda\vert^2\Vert w\Vert_{\h}^{23/12}\Vert \mathcal{A}f\Vert_{\h}^{1/12} = \Vert w\Vert_{\h}^{23/12}\big(\vert\lambda\vert^{24}\Vert \mathcal{A}f\Vert_{\h}\big)^{1/12} \leq \varepsilon\Vert w\Vert_{\h}^2 + C_\varepsilon\vert \lambda \vert^{48}\Vert \mathcal{A}f\Vert_{\h}^2. $$
This implies that \eqref{16equa} holds}. Now, by \eqref{13equa}, \eqref{14equa}, \eqref{16equa} and Young's inequality, we obtain
$$
\left\| w\right\|^2_{\h}\leq\varepsilon\left\| w\right\|^2_{\h}+C_{\varepsilon}|\lambda|^{48}\left\|{\mathcal A} f\right\|^2_{\h} \ \textup{ for any } \ \varepsilon>0.
$$
Hence,
\begin{equation}\label{5-3}
\left\| w\right\|_{\h}\leq C|\lambda|^{24}\left\|{\mathcal A} f\right\|_{\h},
\end{equation}
which shows \eqref{5-2} with $\beta'=1$ and $\beta=24$ for the case $\gamma>0$. Note that \eqref{5-3} implies that $i\lambda - \mathcal{A}$ is injective. In fact, let $w\in D(\mathcal{A})$ with $f:=(i\lambda -\mathcal{A})w = 0$. Then $\mathcal{A}w = i\lambda w\in D(\mathcal{A})$,    which shows $w\in D(\mathcal{A}^2)$, and we can apply \eqref{5-2} to see $w=0$. Therefore  $i\R\cap\sigma({\mathcal A})=\emptyset$. Now, the assertion follows from  \cite{munoz2017racke}, Lemma~5.2.

For the case $\gamma=0$, we can argue analogously and obtain \eqref{5-2} with $\beta'=0$ and $\beta=24$.
\end{proof}

\begin{obser}
a) Note that the proof gives no information on the optimal decay rate.

b) It was shown in \cite{martinez2019regularity}, Theorem 3.2 and Theorem 5.2, that also in the isothermal situation we have polynomial, but no exponential stability if the membrane is undamped.
\end{obser}


\bibliographystyle{abbrv}

\end{document}